\documentclass{amsart}
\usepackage[margin=1.0in]{geometry}
\usepackage{amsthm, amsmath, amssymb, mathtools, mathrsfs}
\usepackage{trimclip, enumitem, adjustbox}
\usepackage{tikz}
\usepackage{tikz-qtree}
\usepackage{aliascnt}
\usepackage{hyperref}
\usepackage{breakurl}
\usepackage{csquotes, version}

\usepackage{bussproofs, longtable}

\makeatletter
\def\LT@start{%
\let\LT@start\endgraf
\endgraf\penalty\z@\vskip\LTpre
\dimen@\pagetotal
\advance\dimen@ \ht\ifvoid\LT@firsthead\LT@head\else\LT@firsthead\fi
\advance\dimen@ \dp\ifvoid\LT@firsthead\LT@head\else\LT@firsthead\fi
\advance\dimen@ \ht\LT@foot
\dimen@ii\vfuzz
\vfuzz\maxdimen
\setbox\tw@\copy\z@
\setbox\tw@\vsplit\tw@ to \ht\@arstrutbox
\setbox\tw@\vbox{\unvbox\tw@}%
\vfuzz\dimen@ii
\advance\dimen@ \ht
\ifdim\ht\@arstrutbox>\ht\tw@\@arstrutbox\else\tw@\fi
\advance\dimen@\dp
\ifdim\dp\@arstrutbox>\dp\tw@\@arstrutbox\else\tw@\fi
\advance\dimen@ -\pagegoal
\ifdim \dimen@>\z@\unskip\vfil\break\fi
\global\@colroom\@colht
\ifvoid\LT@foot\else
\advance\vsize-\ht\LT@foot
\global\advance\@colroom-\ht\LT@foot
\dimen@\pagegoal\advance\dimen@-\ht\LT@foot\pagegoal\dimen@
\maxdepth\z@
\fi
\ifvoid\LT@firsthead\copy\LT@head\else\box\LT@firsthead\fi
\output{\LT@output}}
\makeatother

\emergencystretch=\maxdimen

\usepackage[textsize=footnotesize]{todonotes}

\usepackage{xcolor}
\definecolor{dblue}{rgb}{0,0,0.70}
\hypersetup{
	unicode=true,
	colorlinks=true,
	citecolor=dblue,
	linkcolor=dblue,
	anchorcolor=dblue
}

\makeatletter
\expandafter\g@addto@macro\csname th@plain\endcsname{%
	\thm@notefont{\bfseries}
}%
\expandafter\g@addto@macro\csname th@remark\endcsname{%
	\thm@headfont{\bfseries}
}%
\makeatother


\newtheorem{theorem}{Theorem}[section]	
\newtheorem*{theorem*}{Theorem}

\newaliascnt{lemma}{theorem}
\newtheorem{lemma}[lemma]{Lemma}
\aliascntresetthe{lemma}
\newtheorem*{lemma*}{Lemma}

\newaliascnt{proposition}{theorem}
\newtheorem{proposition}[proposition]{Proposition}
\aliascntresetthe{proposition}

\newaliascnt{corollary}{theorem}
\newtheorem{corollary}[corollary]{Corollary}
\aliascntresetthe{corollary}

\theoremstyle{remark}

\newaliascnt{remark}{theorem}

\aliascntresetthe{remark}
\newaliascnt{question}{theorem}

\aliascntresetthe{question}

\newtheorem*{question*}{Question}

\newaliascnt{definition}{theorem}
\newtheorem{definition}[definition]{Definition}
\aliascntresetthe{definition}

\newaliascnt{example}{theorem}

\aliascntresetthe{example}

\newaliascnt{convention}{theorem}

\aliascntresetthe{convention}

\makeatletter
\def\l@subsection{\@tocline{2}{0pt}{1pc}{5pc}{}} \def\l@subsection{\@tocline{2}{0pt}{2pc}{6pc}{}}
\makeatother
\usepackage[style = numeric, sorting = nyt]{biblatex}
\title{On Separating Wholeness Axioms}
\author{Hanul Jeon}
\email{ \href{mailto:hj344@cornell.edu}{hj344@cornell.edu}}
\urladdr{ \href{https://hanuljeon95.github.io}{https://hanuljeon95.github.io} }
\address{Department of Mathematics, Cornell University, Ithaca, NY 14853}

\mathchardef\mhyphen="2D
\newcommand{\restricts}{\upharpoonright}

\newcommand{\lr}{\leftrightarrow}

\makeatletter
\DeclareRobustCommand\widecheck[1]{{\mathpalette\@widecheck{#1}}}
\def\@widecheck#1#2{%
    \setbox\z@\hbox{\m@th$#1#2$}%
    \setbox\tw@\hbox{\m@th$#1%
       \widehat{%
          \vrule\@width\z@\@height\ht\z@
          \vrule\@height\z@\@width\wd\z@}$}%
    \dp\tw@-\ht\z@
    \@tempdima\ht\z@ \advance\@tempdima2\ht\tw@ \divide\@tempdima\thr@@
    \setbox\tw@\hbox{%
       \raise\@tempdima\hbox{\scalebox{1}[-1]{\lower\@tempdima\box
\tw@}}}%
    {\ooalign{\box\tw@ \cr \box\z@}}}
\makeatother

\newcommand{\dom}{\operatorname{dom}}

\newcommand{\rank}{\operatorname{rank}}

\newcommand{\AC}{\mathsf{AC}}

\newcommand{\ZFC}{\mathsf{ZFC}}

\newcommand{\BTEE}{\mathsf{BTEE}}
\newcommand{\WA}{\mathsf{WA}}
\newcommand{\uWA}{\underline{\mathsf{WA}}}

\newcommand{\Sep}{\mathsf{Sep}}

\newcommand{\Con}{\mathsf{Con}}

\newcommand{\prejiter}{\mathsf{pre}j\mathsf{iter}}
\newcommand{\jiter}{j\mathsf{iter}}
\newcommand{\critseq}{\mathsf{critseq}}

\newcommand{\jht}{j\mhyphen\!\operatorname{ht}}

\newcommand{\uSigma}{\underline{\Sigma}}
\newcommand{\uPi}{\underline{\Pi}}
\newcommand{\uDelta}{\underline{\Delta}}

\newcommand{\sfj}{\mathsf{j}}
\newcommand{\bfG}{\mathbf{G}}

\excludeversion{Fleischmann}

\begin{document}
\maketitle

\begin{abstract}
    In this paper, $\ZFC+\WA_{n+1}$ implies the consistency of $\ZFC+\WA_n$ for $n\ge 0$.
    We also prove that $\ZFC+\WA_n$ is finitely axiomatizable, and $\ZFC+\WA$ is not finitely axiomatizable.
\end{abstract}

\section{Introduction}

The Wholeness axiom, which was proposed by Corazza, is a weakening of the existence of a Reinhardt embedding $j\colon V\to V$ obtained by avoiding the Replacement schema for formulas with the elementary embedding $j$.
More precisely, the Wholeness axiom is formulated over the expanded language $\{\in , j\}$, the language with the membership relationship $\in$ and a unary function for an elementary embedding $j\colon V\to V$, and includes the Separation schema for the expanded language but Replacement schema restricted to formulas \emph{without} $j$.
Unlike that the existence of a Reinhardt embedding is inconsistent with $\AC$, the Wholeness axiom is compatible with $\AC$ modulo large cardinal axioms that are not known to incompatible with $\mathsf{ZFC}$.

Hamkins \cite{Hamkins2001WA} proposed a hierarchy on Wholeness axioms $\WA_n$ by restricting the Separation schema over the expanded language to $\Sigma^j_n$-formulas, $\Sigma_n$ formulas with the new symbol $j$.
He also proved that $\WA_0$ is compatible with $V=\mathrm{HOD}$ and that $\WA_0$ is not equivalent to any of $\WA_n$. 
Hamkins also suggested that each $\WA_n$ is not equivalent to $\WA_m$ for $n<m$, but it was open whether we can separate the Wholeness hierarchy in terms of consistency strength, and whether each $\ZFC+\WA_n$ is finitely axiomatizable.

The main goal of this paper is to establish the strength and the finite axiomatizability of $\WA_n$. In this paper, we will see the following results hold:

\begin{theorem*}[\ref{Corollary: ZFC WA0 consistency}, \ref{Corollary: ZFC WAn consistency}]
    $\ZFC+\WA_{n+1}$ proves the consistency of $\ZFC+\WA_n$ for $n\ge 0$.
\end{theorem*}

\begin{theorem*}[\ref{Corollary: Sect4-WAn finite}, \ref{Corollary: Sect4-WAnotfinite}]
    $\ZFC+\WA_n$ is finitely axiomatizable for each $n$. As a result, $\ZFC+\WA$ is not finitely axiomatizable.
\end{theorem*}

Let us describe the main idea of the proof. Let $j\colon V\to V$ be an elementary embedding witnessing the Wholeness axiom with the critical point $\kappa$.
The first step to show the above result is observing that $V_\kappa$ reflects all first-order formulas of $V$, and the critical sequence $\langle j^n(\kappa)\mid n<\omega\rangle$ is unbounded over the class of ordinals. (See \autoref{Proposition: Sect2-ConsequenceWA0}.)
Hence we may expect to define the truth predicate for specific formulas, and in fact, we can easily see that we can define the truth predicate for formulas without $j$. It boosts the idea that we may build a hierarchy of formulas by considering $\Sigma_\infty$, the class of all formulas without $j$, as atomic formulas, and build up the Levy-like hierarchy with $j$. We may expect the lowest level of the hierarchy, which we call $\Delta^j_0(\Sigma_\infty)$, would have the truth predicate. It turns out that this expectation actually holds.

Then how can we use the truth predicate to establish the consistency? The truth predicate we get is partial, and we cannot directly apply the soundness since it does not satisfy the full first-order logic.
Here we will use proof-theoretic techniques to make use of the partial truth predicate.
We will design sequent calculus $\bfG_{\WA_0}$ for $\ZFC+\WA_0$ without the Cut rule that also proves every consequence of $\ZFC+\WA_0$.
Hence if $\ZFC+\WA_0$ proves the contradiction, we have a cut-free proof of the empty sequent from $\bfG_{\WA_0}$. We can see that every formula in the cut-free proof of the empty sequent is $\Delta^j_0(\Sigma_\infty)$, and the induction argument over the proof shows the partial truth predicate satisfies the falsity, a contradiction.
We need $\ZFC+\WA_1$ (more precisely, $\Pi^j_1$-Induction) since we need to apply induction to a $\Pi^j_1$-sentence, which is impossible over $\ZFC+\WA_0$.
Proving the consistency of $\ZFC+\WA_n$ from $\ZFC+\WA_{n+1}$ follows a similar, but more convoluted argument.

\section*{Acknowledgement}
The author would be grateful for an anonymous referee for pointing out sloppy definitions and arguments and suggesting ideas for better consistency results. Without the reviewer's suggestion, the author could not prove the consistency of $\ZFC+\WA_n$ from $\ZFC+\WA_{n+1}$.
The author would also like to thank the author's advisor, Justin Moore for giving feedback on an earlier version of this paper.

\section{Preliminaries}

Let us clarify our notations before introducing preliminary notions. First-order theories are our central objects, and the definition of first-order expressions is standard. We assume that the formal definition of terms uses a fixed enumeration of variables $v_i$ for $i<\omega$, although we mostly informally denote them by $x$, $y$, $z$, etc. Also, we can always substitute quantified variables so that no quantifier uses the same variable in a single formula. Thus we always assume that each quantifier in a first-order formula bounds a uniquely corresponding variable.
For a function symbol $f$ and a natural number $n$, we use the notation $f^n$ to denote the $n$th iterate of $f$; so for example, $f^2(x)$ denotes the term $f(f(x))$.

We will also construct a `model' of a given first-order theory, and it requires defining substituting variables by elements in a domain. Finite tuples are functions whose domain is a finite ordinal. We denote a finite tuple of elements as vectors like $\vec{a}$. For a tuple of elements $\vec{a}$ and the $i$th variable $v_i$, $\vec{a}_{v_i}$ means the $i$th element of $\vec{a}$. If $i$ is larger than the length of $\vec{a}$, then we take $\vec{a}_{v_i}$ as the empty set. We usually use the notation $\vec{a}_x$ instead of $\vec{a}_{v_i}$, and the reader should understand $\vec{a}_x$ as $\vec{a}_{v_i}$ for some appropriate $i$ such that $x=v_i$.
Also, for a finite tuple $\vec{a}$ and a variable $v_i$, $\vec{a}[b/v_i]$ is a tuple obtained by replacing $i$th element of $\vec{a}$ to $b$. If $i$ is larger than the length of $\vec{a}$, then $\vec{a}[b/v_i]$ is a finite tuple of length $(i+1)$ obtained from $\vec{a}$ by appending enough number of empty sets and $b$ at its $i$th place. We will use the notation $\vec{a}[b/x]$ instead of $\vec{a}[b/v_i]$, and we should understand the former as the latter in a natural manner. Also, we do not distinguish variables and elements unless the distinction is necessary. That is, we write $\phi(a_0,\cdots,a_m)$ or $\phi(\vec{a})$ instead of $\phi(v_0,\cdots,v_m)[\vec{a}]$.
We use $\circ$ for an unknown binary connective of the first-order logic, and $\mathsf{Q}$ for an unspecified quantifier. Bounded quantifiers are often treated as formally distinct objects from unbounded quantifiers.

The distinction between an object theory and a metatheory is often important in this paper, but its distinction might be implicit. We sometimes call formulas in the object theory as codes for formulas to distinguish them from actual formulas over the metatheory. Also, we often call natural numbers on the metatheory as standard natural numbers. A useful rule-of-thumb to note is that natural numbers to denote the complexity of formulas (like $n$ in $\Sigma_n$) in this paper are always standard.

\subsection{Various hierarchies of formulas}
In this subsection, we define various types of hierarchy on the formula of set theory. Most set theorists work with the usual Levy hierarchy, with the following definition:
\begin{definition}
    $\Sigma_n$ and $\Pi_n$-formulas are defined recursively as follows:
    \begin{itemize}
        \item $\Sigma_0=\Pi_0$ are the class of formulas with no unbounded quantifiers.
        \item A formula is $\Pi_n$ if it is of the form  $\forall x \phi$, where $\phi$ is $\Sigma_{n-1}$.
        \item A formula is $\Sigma_n$ if it is of the form  $\exists x \phi$, where $\phi$ is $\Pi_{n-1}$.
    \end{itemize}
\end{definition}
For a given theory $T$ and a formula $\phi$, $\phi$ is \emph{$\Sigma_n$ or $\Pi_n$ over $T$}, or \emph{$T$-provably $\Sigma_n$ or $\Pi_n$} if $T$ proves $\phi$ is equivalent to some $\Sigma_n$ or $\Pi_n$ formula respectively.
If $T$ proves $\phi$ is equivalent to both of a $\Sigma_n$ and a $\Pi_n$ formula, then we say $T$ is \emph{$\Delta_n$ over $T$.}
We omit $T$ and simply say a formula is $\Sigma_n$, $\Pi_n$, or $\Delta_n$ if the context is clear. Also, we call $\Sigma_0$-formulas as $\Delta_0$-formulas for convenience.

It is well-known folklore that repeating the same type of quantifiers does not affect being $\Sigma_n$ or $\Pi_n$; For example, for a bounded formula $\phi(x,y)$, $\exists x\exists y\phi(x,y)$ is equivalent to a $\Sigma_1$-formula although it is not literally $\Sigma_1$.
Let us formulate this folklore into a solid language for future use. The point of the subsequent lemmas is that they are provable over very weak subsystem of $\mathsf{ZFC}$, namely, \emph{bounded Zermelo set theory without powerset} $\mathsf{Z}_0^-$, comprising Extensionality, Foundation, Pairing, Union, Infinity, and Separation for bounded formulas.
The following results are presented in \cite{FriedmanLiWong2016KPalpharec}, and similar proofs for them appear in \cite{JeonWalshOrdinalAnalysis}, but we reproduce the results and their proofs for future use.

\begin{lemma}[$\mathsf{Z}_0^-$] \label{Lemma: Sect2-Basics}
    The following formulas are all equivalent to a $\Delta_0$-formula:
    \begin{enumerate}
        \item $z=\{x,y\}$.
        \item $\{x,y\}\in z$.
        \item $z=\langle x,y\rangle$, where $\langle x,y\rangle = \{\{x\},\{x,y\}\}$ is a Kuratowski ordered pair.
        \item $z = \bigcup x$.
        \item $z = x\cup y$.
        \item $z = \varnothing$.
    \end{enumerate}
    Furthermore, the Kuratowski definition of ordered pair is correct in the following sense: If $\langle x_0,x_1\rangle = \langle y_0,y_1\rangle$, then $x_0=y_0$ and $x_1=y_1$.
\end{lemma}
\begin{proof}
    We can see that the following equivalences hold:
    \begin{enumerate}
        \item $z = \{x,y\} \iff (x\in z \land y\in z) \land \forall w\in z(w=x\lor w=y)$.
        \item $\{x,y\}\in z \iff \exists w\in z (w = \{x,y\})$.
        \item $z = \langle x,y\rangle \iff (\{x\}\in z\land \{x,y\}\in z)\land \forall w\in z (w = \{x\} \lor w = \{x,y\})$.
        \item $z=\bigcup x \iff [\forall w\in z\exists y\in x (w\in y)] \land [\forall y\in x\forall w\in y (w\in z)]$.
        \item $z=x\cup y \iff [\forall w\in x(w\in z)\land \forall w\in y(w\in z)] \land [\forall w\in z(z\in x \lor z\in y)]$.
        \item $z=\varnothing \iff \forall w\in z (w\neq w)$.
    \end{enumerate}
    Now let us turn our view to the correctness of Kuratowski ordered pair.
    Suppose that $\langle x_0,x_1\rangle = \langle y_0,y_1\rangle$, that is, 
    \begin{equation*}
        \{\{x_0\},\{x_0,x_1\}\} = \{\{y_0\},\{y_0,y_1\}\}.
    \end{equation*}
    Then we have $\{x_0\}=\{y_0\}\lor \{x_0\}=\{y_0,y_1\}$.
    Similarly, we have  $\{x_0,x_1\}=\{y_0\}\lor \{x_0,x_1\}=\{y_0,y_1\}$. Based on this, let us divide the cases:
    \begin{itemize}
        \item If $\{x_0\} = \{y_0,y_1\}$, then $x_0=y_0=y_1$. Thus we get $\{x_0,x_1\}=\{y_0\}$, so $x_1=y_0$. In sum, $x_0=x_1=y_0=y_1$.

        \item If $\{x_0,x_1\}=\{y_0\}$, then by an argument similar to the previous one, we have $x_0=x_1=y_0=y_1$.

        \item Now consider the case $\{x_0\}=\{y_0\}$ and $\{x_0,x_1\}=\{y_0,y_1\}$. Then we have $x_0=y_0$, so $\{x_0,x_1\}=\{x_0,y_1\}$. Hence $x_1\in \{x_0,y_1\}$, so either $x_1=x_0$ or $x_1=y_1$.
        But if $x_1=x_0$, then from $y_1\in \{x_0,x_1\}$ we get $x_1=y_1$. \qedhere
    \end{itemize}
\end{proof}

\begin{lemma}[$\mathsf{Z}_0^-$] \label{Lemma: Sect2-orderedpair-property}
    There are $\Delta_0$-definable functions $\pi_0$ and $\pi_1$ such that $\pi_0(\langle x,y\rangle)=x$ and $\pi_1(\langle x,y\rangle)=y$.
    
    Furthermore, if $\phi(x)$ is a $\Delta_0$-formula, then so is $\phi(\langle x,y\rangle)$, $\phi(\pi_0(x))$, and $\phi(\pi_1(x))$.
\end{lemma}
\begin{proof}
    Consider the following functions:
    \begin{itemize}
        \item $\pi_0(x)=\bigcup\{ u\in\bigcup x\mid \forall v\in x (u\in v)\}$.
        \item $\pi_1(x) = \bigcup\{u\in \bigcup x \mid \exists v\in \bigcup x \forall p\in x [p=\{v\}\lor p=\{u,v\}]\}$.
    \end{itemize}
    Clearly, $\mathsf{Z}_0^-$ proves the above functions are well-defined. Moreover, we can see that the formulas $\pi_0(x)=y$ and $\pi_1(x)=y$ are equivalent to $\Delta_0$-formulas: First, we can see
    \begin{align*}
        z\in \pi_0(x) &\iff \exists w [y\in w_0 \land w_0\in \mathop{\textstyle \bigcup} x \land \forall v\in x(y\in v)] \\ 
        &\iff \exists w [y\in w_0 \land (\exists w_1\in w_0 (w_0\in w_1)) \land \forall v\in x(y\in v)] \\
        &\iff \exists w_1 \in x\exists w_0\in w_1 (y\in w\land\forall v\in x (y\in v)),
    \end{align*}
    which is $\Delta_0$, and we can see
    \begin{align*}
        \pi_0(x)\subseteq y &\iff \forall z (z\in \pi_0(x)\to z\in y) \\
        &\iff \forall z \left[ [\exists w_1 \in x\exists w_0\in w_1 (y\in w\land\forall v\in x (y\in v))]\to z\in y\right]\\
        &\iff \forall z\forall w_1\in x\forall w_0\in w_1 [(z\in w_0\land \forall v\in x(z\in v))\to z\in y] \\
        &\iff \forall w_1\in x\forall w_0\in w_1\forall z\in w_0 [\forall v\in x(z\in v))\to z\in y].
    \end{align*}
    and 
    \begin{equation*}
        \pi_0(x)=y \iff \pi_0(x)\subseteq y \land \forall z\in y(z\in \pi_0(x)),
    \end{equation*}
    so $\pi_0(x)=y$ is equivalent to a $\Delta_0$-formula.
    In a similar way, but with a longer proof, we can see that $\pi_1(x)=y$ is equivalent to a $\Delta_0$-formula.

    The second part follows from induction on a bounded formula $\phi$. We may also prove it by using the fact that ordered pair functions and projections $\pi_0$, $\pi_1$ are rudimentary, hence simple. (See \cite{Jensen1972Fine}). 
\end{proof}

\begin{lemma}[$\mathsf{Z}_0^-$]\label{Lemma: Sect2-TwoUniversalReducing}
    Let $\phi_0(x_0,x_1,v_0,\cdots, v_{n-1})$ be a $\Delta_0$-formula. Then $\phi_0(\pi_0(x),\pi_1(x),v_0,\cdots, v_{n-1})$ is also $\Delta_0$ and 
    \begin{equation*}
        \forall x_0\forall x_1\exists v_0\cdots\mathsf{Q} v_{n-1} \phi_0(x_0,x_1,v_0,\cdots, v_{n-1}) \lr \forall x \exists v_0\cdots\mathsf{Q} v_{n-1} \phi_0(\pi_0(x),\pi_1(x),v_0,\cdots, v_{n-1})
    \end{equation*}
\end{lemma}
\begin{proof}
    The first part follows from \autoref{Lemma: Sect2-orderedpair-property}. The last part is easy to prove.
\end{proof}

From the previous lemma, we get the following:
\begin{lemma}[$\mathsf{Z}_0^-$]\label{Corollary: Sect2-PrenexFormReducing} \pushQED{\qed}
    If $\phi$ and $\psi$ are $\Sigma_n$, then $\exists x\phi(x)$, $\phi\land\psi$, $\phi\lor\psi$ are also equivalent to a $\Sigma_n$-formula. 
    
    Similarly, if $\phi$ and $\psi$ are $\Pi_n$, then $\forall x\phi(x)$, $\phi\land\psi$, $\phi\lor\psi$ are also equivalent to a $\Pi_n$ formula.
\end{lemma}
\begin{proof}
    Let us only consider the case when $\phi$ and $\psi$ are $\Pi_n$. The case for $\Sigma_n$ is symmetric.

    We first claim that if $\phi(x)$ is $\Pi_n$, then $\forall x\phi(x)$ is equivalent to a $\Pi_n$-formula.
    Suppose that $\phi(x)$ is of the form 
    \begin{equation*}
        \forall v_0\exists v_1\cdots\mathsf{Q} v_{n-1} \phi_0(x,v_0,\cdots,v_{n-1})
    \end{equation*}
    for some $\Delta_0$-formula $\phi_0$.
    By \autoref{Lemma: Sect2-TwoUniversalReducing}, we have
    \begin{equation*}
        \forall x\forall v_0\exists v_1\cdots\mathsf{Q} v_{n-1} \phi_0(x,v_0,v_1,\cdots, v_{n-1}) \lr \forall x \exists v_1\cdots\mathsf{Q} v_{n-1} \phi_0(\pi_0(x),\pi_1(x),v_1,\cdots, v_{n-1}).
    \end{equation*}

    Next, let us prove that $\phi\land\psi$ is equivalent to a $\Pi_n$-formula.  The case for $\phi\lor\psi$ is analogous.
    Since $\phi$ and $\psi$ are $\Pi_n$, we have $\Delta_0$-formulas $\phi_0$ and $\psi_0$ such that
    \begin{equation} \label{Formula: Sect2-Phi-Pin}
        \phi(\vec{x}) \equiv \forall u_0 \exists  u_1 \cdots \mathsf{Q} u_{n-1} \phi_0(\vec{x}, u_0,u_1,\cdots, u_{n-1})
    \end{equation}
    and
    \begin{equation} \label{Formula: Sect2-Psi-Pin}
        \psi(\vec{x}) \equiv \forall v_0 \exists v_1 \cdots \mathsf{Q} v_{n-1} \psi_0(\vec{x}, v_0,v_1,\cdots, v_{n-1}).
    \end{equation}
    We may assume that all of $u_i$ and $v_i$ are different variables by substitution, so by predicate-logic level computation, we have
    \begin{equation*}
        \phi(\vec{x})\land\psi(\vec{x}) \iff \\
        \forall u_0\forall v_0 \exists u_1\exists v_1\cdots \mathsf{Q}u_{n-1}\mathsf{Q}v_{n-1}[\phi_0(\vec{x}, u_0,u_1,\cdots, u_{n-1})\land \psi_0(\vec{x},, v_0,v_1,\cdots, v_{n-1})].
    \end{equation*}
    Now inductively reduce duplicated quantifiers from the inside to the outside by applying \autoref{Lemma: Sect2-TwoUniversalReducing}.
    As a result, we can see that $\phi\land\psi$ is equivalent to
    \begin{equation*}
        \forall w_0\exists w_1\cdots \mathsf{Q} w_{n-1} [\phi_0(\vec{x}, \pi_0(w_0), \pi_0(w_1),\cdots, \pi_0(w_{n-1}))\land \psi_0(\vec{x}, \pi_1(w_0), \pi_1(w_1),\cdots, \pi_1(w_{n-1}))]. \qedhere 
    \end{equation*}
\end{proof}

Observe that the above proof works uniformly on formulas, so we can see the following meta-corollary holds:
\begin{corollary}\label{Corollary: Sect2-Effective-normalization} \pushQED{\qed}
    Suppose that $\phi(x)$ and $\psi(x)$ are $\Sigma_n$-formulas.
    Then we can effectively decide $\Sigma_n$-formulas $D_{\exists x\phi}$, $D_{\phi\land\psi}$, and $D_{\phi\lor\psi}$ such that $\mathsf{Z}^-_0$ proves $\exists x\phi(x)\lr D_{\exists x\phi(x)}$, $\phi\land\psi\lr D_{\phi\land\psi}$, and $\phi\lor\psi\lr D_{\phi\lor\psi}$.
    
    Similarly, if $\phi(x)$ and $\psi(x)$ are $\Pi_n$-formulas, then we can effectively decide $\Pi_n$-formulas $D_{\forall x\phi}$, $D_{\phi\land\psi}$, and $D_{\phi\lor\psi}$ such that $\mathsf{Z}^-_0$ proves $\forall x\phi(x)\lr D_{\forall x\phi(x)}$, $\phi\land\psi\lr D_{\phi\land\psi}$, and $\phi\lor\psi\lr D_{\phi\lor\psi}$.
    \qedhere 
\end{corollary}
Its proof follows from observing the proof of \autoref{Corollary: Sect2-PrenexFormReducing}. Let us provide the corresponding $D$-formulas when $\phi$ and $\psi$ are $\Pi_n$ for future use: Let $\phi(\vec{x})=\phi(x_0,\cdots,x_m)$ and $\psi(\vec{x}) = \psi(x_0,\cdots, x_m)$ be $\Pi_n$-formulas as stated in the proof of \autoref{Corollary: Sect2-PrenexFormReducing}, \eqref{Formula: Sect2-Phi-Pin} and \eqref{Formula: Sect2-Psi-Pin}. Then
\begin{itemize}
    \item $D_{\forall x_0 \phi}$ is $\forall w \exists u_1\cdots \mathsf{Q} u_{n-1}\phi_0(\pi_0(w),x_1,\cdots,x_m,\pi_1(w),u_1,\cdots, u_{n-1})$
    \item $D_{\phi\land\psi}(\vec{x})$ is  
    $\forall w_0\exists w_1\cdots \mathsf{Q} w_{n-1} [\phi_0(\pi_0(w_0), \cdots, \pi_0(w_{n-1}))\land \psi_0(\pi_1(w_0), \pi_1(w_{n-1}))].$
    \item $D_{\phi\lor\psi}(\vec{x})$ is  
    $\forall w_0\exists w_1\cdots \mathsf{Q} w_{n-1} [\phi_0(\pi_0(w_0), \cdots, \pi_0(w_{n-1}))\lor \psi_0(\pi_1(w_0), \pi_1(w_{n-1}))].$
\end{itemize}

\subsection{Hierarchies of formulas with $j$}
Due to the subtle nature of the Wholeness axiom, we should distinguish formulas with the elementary embedding symbol $j\colon V\to V$ and without the elementary embedding symbol. Following the notation introduced in \cite{MatthwesPhD} and \cite{JeonMatthews2022}, we distinguish formulas with $j$ and those without $j$.

\begin{definition}
    We call $\Sigma_n$ or $\Pi_n$ formuals with $j$ as $\Sigma_n^j$ or $\Pi_n^j$ respectively. $\Sigma_n$ and $\Pi_n$ mean class of formulas without $j$.
\end{definition}

\begin{Fleischmann}
A major drawback of the usual Levy hierarchy is that there is no reason to believe each of $\Sigma_n$ and $\Pi_n$ are provably closed under bounded quantifiers. To be precise, $\ZFC$ proves that we can `alter' the order of a bounded $\forall$ and an unbounded $\exists$ in the sense that
\begin{equation*}
    \forall x\in a\exists y \phi(x,y)\to \exists b \forall x\in a\exists y\in b\phi(x,y).
\end{equation*}
We can see that the above formula is exactly what the axiom of Collection claims. However, we do not have Collection for formulas with $j$. Thus adding the bounded quantifier to a $j$-formula may increase its complexity unlike what happens over the usual $\ZFC$ context. A related issue about this complexity issue is that the claim that a $\Sigma^j_n$-definable function is really a function is no longer $\Sigma^j_n$-describable. 
Thus we define a new hierarchy by including bounded quantifiers by following the definition provided in \cite{JeonMatthews2022}.

\begin{definition}
    We define the \emph{Levy-Fleischmann hierarchy} inductively as follows:
    \begin{itemize}
        \item $\uDelta_0=\uSigma_0=\uPi_0$ is the collection of all bounded formulas.
        \item $\uSigma_n$ is the least collection of formulas containing $\uPi_{n-1}$ closed under $\land$, $\lor$, bounded quantifiers, and unbounded $\exists$.
        \item $\uPi_n$ is the least collection of formula containing $\uSigma_{n-1}$ closed under $\land$, $\lor$, bounded quantifiers, unbounded $\forall$, and formulas of the form $\phi\to\psi$, where $\phi$ is $\uSigma_{n-1}$ and $\psi$ is $\uPi_n$.
    \end{itemize}

    We call a formula $\uDelta_n$ if it is both of $\uSigma_n$ and $\uPi_n$.
\end{definition}
Levy-Fleischmann hierarchy was defined in the intuitionistic context, so it includes the implication in the definition of $\uPi_n$. Under the classical context, we may drop the condition for the implication in the definition of $\uPi_n$, and it will result in the equivalent definition. It is easy to see that $\uSigma_1$ formulas are also known as $\Sigma$-formulas.
Also, remind that all formulas over the language of set theory are some of $\uSigma_n$ or $\uPi_n$ for some $n$. Furthermore, we will see later that $\uSigma_n$-formulas are closed under recursive definitions.
We also use the notation $\uSigma_n^j$, $\uPi_n^j$, and $\uDelta_n^j$ to distinguish $j$-formulas from formulas without $j$.
\end{Fleischmann}

Lastly, for a given class of formulas $\Gamma$, we define $\Sigma_n(\Gamma)$, $\Pi_n(\Gamma)$ 
as the Levy hierarchy starting from $\Gamma$. That is, $\Delta_0(\Gamma)=\Sigma_0(\Gamma)=\Pi_0(\Gamma)$ is the least class of formulas containing $\Gamma$ and closed under logical connectives and bounded quantifiers, and we define $\Sigma_n(\Gamma)$ and $\Pi_n(\Gamma)$ similarly.
We also use notation like $\Sigma_n^j(\Gamma)$ to distinguish formulas with $j$. Especially, $\Delta^j_0(\Sigma_\infty)$ is the least class of formulas containing $\Sigma_\infty$-formulas, that is, formulas without $j$, which is closed under logical connectives and quantifiers bounded by terms.
\begin{definition}
    $\Delta^j_0(\Sigma_\infty)$ is the least class of formulas satisfying the following clauses:
    \begin{itemize}
        \item If $\phi(x_0,\cdots,x_{m-1})$ is $\Sigma_\infty$ and $t_0,\cdots, t_{m-1}$ are terms, then $\phi(t_0,\cdots,t_{m-1})$ is $\Delta^j_0(\Sigma_\infty)$.
        \item If $\phi$ and $\psi$ are $\Delta^j_0(\Sigma_\infty)$, then so are $\lnot\phi$ and $\phi\circ\psi$, where $\circ$ is either $\land$, $\lor$, or $\to$.
        \item If $\phi(x)$ is $\Delta^j_0(\Sigma_\infty)$ and $t$ is a term, then so is $\mathsf{Q} x\in t \phi(x)$.
    \end{itemize}
    Note that every term in the language takes the form $j^n(y)$ for some (meta-)natural $n$ and a free variable $y$.
    We define $\Sigma^j_n(\Sigma_\infty)$ and $\Pi^j_n(\Sigma_\infty)$ inductively as like we define $\Sigma_n$ and $\Pi_n$.
\end{definition}

Let us remark that we can reduce the repeating quantifiers in $\Sigma^j_n$- and $\Pi^j_n$-formulas similar to what \autoref{Lemma: Sect2-TwoUniversalReducing} state and the reduction procedure is uniform.
Their proof is similar to that of \autoref{Lemma: Sect2-TwoUniversalReducing}, \autoref{Corollary: Sect2-PrenexFormReducing}, and \autoref{Corollary: Sect2-Effective-normalization} so we only provide their statements.
Furthermore, the proof of the following lemmas does not use any properties of $j$, so they are provable over $\mathsf{Z}_0^-$.

\begin{lemma}[$\mathsf{Z}^-_0$] \pushQED{\qed}\label{Lemma: Sect2-jPrenexFormReducing}
    If $\phi$ and $\psi$ are $\Sigma^j_n$, then $\exists x\phi(x)$, $\phi\land\psi$, $\phi\lor\psi$ are also equivalent to a $\Sigma^j_n$-formula. 
    
    Similarly, if $\phi$ and $\psi$ are $\Pi^j_n$, then $\forall x\phi(x)$, $\phi\land\psi$, $\phi\lor\psi$ are also equivalent to a $\Pi^j_n$ formula.
    \qedhere 
\end{lemma}

\begin{corollary}\label{Corollary: Sect2-jprenex-Effective-normalization} \pushQED{\qed}
    Suppose that $\phi(x)$ and $\psi(x)$ are $\Sigma^j_n$-formulas.
    Then we can effectively decide $\Sigma^j_n$-formulas $D_{\exists x\phi}$, $D_{\phi\land\psi}$, and $D_{\phi\lor\psi}$ such that $\mathsf{Z}^-_0$ proves $\exists x\phi(x)\lr D_{\exists x\phi(x)}$, $\phi\land\psi\lr D_{\phi\land\psi}$, and $\phi\lor\psi\lr D_{\phi\lor\psi}$.
    
    Similarly, if $\phi(x)$ and $\psi(x)$ are $\Pi^j_n$-formulas, then we can effectively decide $\Pi^j_n$-formulas $D_{\forall x\phi}$, $D_{\phi\land\psi}$, and $D_{\phi\lor\psi}$ such that $\mathsf{Z}^-_0$ proves $\forall x\phi(x)\lr D_{\forall x\phi(x)}$, $\phi\land\psi\lr D_{\phi\land\psi}$, and $\phi\lor\psi\lr D_{\phi\lor\psi}$.
    \qedhere 
\end{corollary}

\subsection{$\BTEE$ and $\WA$}
In this subsection, we review the \emph{Basic Theory of Elementary Embeddings} $\BTEE$, the \emph{Wholeness Axiom} $\WA$ and their consequences.
Most of the definitions and results in this subsection are due to \cite{Corazza2006}.
We start from a `minimal' theory for an elementary embedding $j\colon V\to V$, named  the \emph{Basic Theory of Elementary Embeddings} $\BTEE$.
\begin{definition}
    $\BTEE$ is a combination of the following claims:
    \begin{itemize}
        \item Elementarity: for any formula $\phi(\vec{x})$ without $j$, we have $\forall \vec{a}[\phi(\vec{a})\to\phi(j(\vec{a}))]$.
        \item Critical Point: $j$ has a critical point, that is, there is an ordinal $\alpha$ such that $j(\alpha)>\alpha$ and $j(\xi)=\xi$ for all $\xi<\alpha$.
    \end{itemize}
\end{definition}

$\BTEE$ is so weak that it even does not allow applying induction on $\omega$ to formulas with $j$. Thus we postulate induction over $\omega$ for $j$-formulas as a separate claim. Also, we consider Separation restricted to formulas in $\Gamma$ to gauge the strength of restricted Separation:
\begin{definition}
    Let $\Gamma$ be a class of formulas. Then $\Gamma$-Induction is induction schema on $\omega$ for formulas in $\Gamma$. Similarly, $\Gamma$-Separation is a Separation schema restricted to formulas in $\Gamma$.
\end{definition}
It is easy to see that $\Gamma$-Separation shows not only $\Gamma$-Induction but also induction for the negation of formulas in $\Gamma$.

We sometimes consider $\Delta_n$ properties, which is both $\Sigma_n$ and $\Pi_n$. Unlike the latter two, there is no single entity of $\Delta_n$-formulas that is independent of the background theory. Thus defining axioms restricted to $\Delta_n$-properties requires a different formulation.
\begin{definition}
    $\Delta^j_n$-Separation is the following statement: For each $\Sigma^j_n$-formula $\phi(x)$ and $\Pi^j_n$-formula $\psi(x)$, a set $a$ and a parameter $p$, 
    if we know $\forall x\in a (\phi(x,p)\lr \psi(x,p))$, then $\{x\in a\mid \phi(x,p)\}$ exists for any $a$.
    We can define $\Delta^j_n$-Induction in a similar manner.
\end{definition}

With an appropriate induction scheme, we can show we can recursively define a function over the set of natural numbers:
\begin{theorem}[{\cite[Theorem 4.6]{Corazza2006}}]
\label{Theorem: Sect2-Corazza-Recursion-n+2}
    Suppose that $F\colon V\to V$ is a $\Sigma^j_n$-definable class function with parameter $p$, that is, there is a $\Sigma^j_n$-formula with parameter $p$ expressing $y= F(x)$.
    \begin{enumerate}
        \item If we assume $\Sigma^j_{n+2}$-Induction, then there is a $\Sigma^j_{n+2}$-definable $G\colon \omega\to V$ with parameter $p$ such that $G\restricts k$ is a set and $G(k)=F(G\restricts k)$ for each $k<\omega$.
        \item If $n=0$, then $\Sigma^j_1$-Induction shows we can find $\Sigma^j_1$-definable $G$ with parameter $p$. 
    \end{enumerate}
    The choice of $G$ does not depend on the parameter $p$.
\end{theorem}
For the later analysis, let us provide the proof for the above theorem:
\begin{proof}
    Let $\phi(x,y,p)$ be a $\Sigma^j_n$-formula defining $F(x)=y$. 
    Now let $\gamma(g,m,p)$ be the following formula:
    \begin{equation*}
        \gamma(g,m,p)\equiv \text{``$g$ is a function''} \land \dom g = m \land \forall i\in m \phi(g\restricts i, g(i),p).
    \end{equation*}
    If $\phi(x,y,p)$ is $\Sigma^j_n$, then $\gamma(g,m,p)$ is $\Pi^j_{n+1}$ due to the universal quantifier over $i$.
    Now we will claim the following statements: For every natural number $m$,
    \begin{itemize}
        \item If $\gamma(f,m,p)$ and $\gamma(g,m,p)$, then $f=g$.
        \item There is $g$ such that $\gamma(g,m,p)$.
    \end{itemize}
    The first claim follows from applying induction to the formula $i<m \to f(i)=g(i)$, which is $\Delta_0$.
    The second claim follows from applying induction on $m$ to the formula $\exists g \gamma(g,m,p)$, which is $\Sigma^j_{n+2}$. Now define $G(m)$ as follows:
    \begin{equation*}
        G(m)=y \iff \exists g [\gamma(g,m+1,p)\land g(m)=y].
    \end{equation*}
    If $\phi(x,y,p)$ is $\Delta^j_0$, then $\gamma(g,m,p)$ is also $\Delta^j_0$.
    Hence the formula $\exists g \gamma(g,m,p)$ is $\Sigma^j_1$, so $\Sigma^j_1$-Induction suffices to prove $\forall m<\omega\exists g \gamma(g,m,p)$. Finally, the defining formula for $G$ is $\Sigma^j_1$ with parameter $p$. It is clear that the defining formula of $G$ is uniform in $p$.
\end{proof}
The reader might be curious why the definitional complexity of $G$ is $\Sigma^j_{n+2}$ and not $\Sigma^j_n$.
We will show a trick on how to reduce the complexity of $G$.
However, let us point out that our naive $\mathsf{ZFC}$-minded argument does not work: It is unclear whether bounded quantifiers do not affect the complexity of formulas. If we work over $\mathsf{ZFC}$, then the axiom of Collection
\begin{equation*}
    \forall x\in a\exists y\phi(x,y)\to \exists b\forall x\in a\exists y\in b \phi(x,y)
\end{equation*}
provides a way to reduce bounded quantifiers. However, we do not have Collection for $j$-formulas.

But in this case, we can reduce bounded quantifiers in a different way. Let us observe that the bounded quantifier that increases the complexity of the formula $G(m)=y$ is $\forall i\in m$, which is bounded by a natural number.
We can see that under $\Sigma^j_n$-Induction, adding a quantifier bounded by a natural number to a $\Sigma^j_n$-formula does not increase the complexity of a $\Sigma^j_n$-formula:
\begin{lemma}[$\mathsf{Z}^-_0+\Sigma^j_n$-Induction]
    Let $\phi(i,y)$ be a $\Pi^j_{n-1}$-formula Then we can prove the following: For a natural number $m<\omega$,
    \begin{equation*}
        \forall i<m \exists y\phi(i,y)\to \exists f[(\text{$f$ is a function}) \land \dom f=m\land \forall i<m \phi(i,f(i))].
    \end{equation*}
\end{lemma}
\begin{proof}
    Suppose that $\forall i<m\exists y\phi(i,y)$ holds. We apply induction to the following formula:
    \begin{equation*}
        k\le m\to \exists f (\text{$f$ is a function}) \land \dom f=k\land \forall i<k \phi(i,f(i)).
    \end{equation*}
    We can see that the above formula is equivalent to a $\Sigma^j_n$-formula, so we can apply $\Sigma^j_n$-Induction. The case $k=0$ is easy, and now assume that the above statement holds for a given $k$. 
    If $k+1\le m$, then $k<m$, so we can find $y$ such that $\phi(k,y)$. By the inductive hypothesis, we have a function $f\colon k\to V$ such that $\phi(i,f(i))$ holds for all $i<k$. If we take $f'=f\cup\{(k,y)\}$, then $f'$ is the desired function we find.
\end{proof}
By \autoref{Lemma: Sect2-jPrenexFormReducing}, a conjunction of $\Pi^j_{n-1}$ formulas is equivalent to a $\Pi^j_{n-1}$ formula. Hence we get
\begin{corollary}[$\mathsf{Z}^-_0+\Sigma^j_n$-Induction]\pushQED{\qed} \label{Corollaty: Sect2-BddQreduction}
    Let $\phi(i,y)$ be a $\Sigma^j_n$-formula and $m<\omega$. Then $\forall i<m\phi(i,y)$ is equivalent to a $\Sigma^j_n$-formula. \qedhere 
\end{corollary}

Going back to the proof of \autoref{Theorem: Sect2-Corazza-Recursion-n+2}, we can see that if $F$ is defined by a $\Sigma^j_n$-formula and we have $\Sigma^j_n$-Induction, then the resulting $G$ is $\Sigma^j_n$-definable. 
Furthermore, if $F$ is $\Delta^j_n$-definable, then $\gamma(g,m)$ becomes a $\Delta^j_n$-formula. We still need a $\Sigma^j_n$-Induction to ensure $\exists g \gamma(g,m)$ holds for all $m$, but by uniqueness of $g$ satisfying $\gamma(g,m)$, the following formula also defines $G$, which is $\Pi^j_n$:
\begin{equation*}
    \forall g [\gamma(g,m+1)\to g(m)=y].
\end{equation*}
Hence if $F$ is $\Delta^j_n$-definable, then $G$ is also $\Delta^j_n$-definable.
In sum, we have the following:
\begin{corollary}[$\mathsf{Z}^-_0+\Sigma^j_n$-Induction]\label{Corollary: Sect2-optimized-recursion} \pushQED{\qed}
    Let $F\colon V\to V$ be a $\Sigma^j_n$-definable function with parameter $p$. Then the function $G\colon \omega\to V$ satisfying $G(k)=F(G\restricts k)$ is $\Sigma^j_n$-definable with parameter $p$. 
    
    Furthermore, if $F$ is $\Delta^j_n$-definable with parameter $p$, then $G$ is also $\Delta^j_n$-definable with parameter $p$. The defining formula of $G$ is uniform in $p$. \qedhere 
\end{corollary}

\begin{Fleischmann}
We can mend this issue if we work with the Levy-Fleischmann hierarchy instead of the usual Levy hierarchy. If follows from by examining the usual proof of recursion theorem (or that of Theorem 4.6 of \cite{Corazza2006}.)

\begin{theorem} \pushQED{\qed}
    Let $F\colon V\to V$ be a $\uSigma^j_n$-definable class function. If we assume $\uSigma^j_n$-Induction, then we can show there is a class function $G\colon \omega\to V$ such that $G\restricts n$ is a set for each $n<\omega$ and $G(n) = F(G\restricts n)$.
    \qedhere 
\end{theorem}
\end{Fleischmann}

Now let us define an abbreviation for some formulas, which is related to iterating $j$. 
\begin{definition}
    Let $\kappa$ be a critical point of $j\colon V\to V$.
    Now consider the following formulas:
    \begin{itemize}
        \item $\prejiter(f,n,x)\equiv \text{``$f$ is a function''} \mathrel{\land} (\dom f = n+1) \mathrel{\land}  (f(0) = x) \mathrel{\land} \forall i\in\omega [0<i\le n\to f(i) = j(f(i+1))].$ 
        \item $\jiter(n,x,y) \equiv n\in\omega \to \exists f [\prejiter(f,n,x)\land f(n)=y]$.
        \item $\critseq(n,y) \equiv \jiter(n,\kappa,y)$.
    \end{itemize}
    We write $y=j^n(\kappa)$ if $\critseq(n,y)$ holds, and this notation will be justified later.
\end{definition}
Corazza \cite[(2.7)--(2.9)]{Corazza2006} called each of the above formulas $\Theta$, $\Phi$, $\Psi$ respectively with a different formulation of $\Phi$: Corazza did not include $f(n)=y$ as a part of $\Phi$ (or $\jiter$ under our notation.)

We can show the following results:
\begin{proposition} \pushQED{\qed}
$\ZFC + \BTEE$ proves the followings: 
    \begin{enumerate}
        \item (\cite[Lemma 2.16]{Corazza2006}) For each $n\in\omega$, $x$, and $y$, if $\prejiter(f,n,x)$, $x$ is an ordinal, and $y=f(n)$, then $y$ is also an ordinal.
        \item (\cite[Lemma 2.17]{Corazza2006}) $\critseq(n,y)$ defines a strictly increasing sequence of ordinals, although the resulting sequence may not be total.
        \item (\cite[Lemma 2.18]{Corazza2006}) $\forall n\in\omega\forall x\forall y\ [\jiter(n,j(x),y)\to \jiter(n+1,x,y)]$
        \item (\cite[Lemma 2.19]{Corazza2006}) $\forall n\in\omega\forall x\forall y\ [\jiter(n,x,y)\to\jiter(n,j(x),j(y))]$.
        \item (\cite[Proposition 4.4(1)]{Corazza2006}) $\forall f\forall g\forall x\forall n\in\omega\ [\jiter(f,n,x)\land \jiter(g,n,x)\to f=g]$.
        \qedhere 
    \end{enumerate}
\end{proposition}

Also, if we have $\Sigma^j_0$-Induction, then we can prove the following:
\begin{proposition}[{\cite[Proposition 5.3]{Corazza2006}}] \pushQED{\qed} \label{Proposition: Sect2-ConsequencesSigmaj0Ind}
    $\ZFC+\BTEE$ with $\Sigma^j_0$-Induction proves the following: Let $\kappa$ be a critical point of $j\colon V\to V$.
    \begin{enumerate}
        \item For every $n,m\in\omega$, if $j^n(\kappa)$ exists and $m\le n$, then $j^m(\kappa)$ exists.
        \item For every $n\in\omega$ and an ordinal $\alpha$, if $j^n(V_\alpha)$ exists, then $j^n(x)$ exists for all $x\in V_\alpha$.
        \item For every $n\in\omega$ and $x$, if $j^n(x)$ exists, then $j^n(\rank x)$ also exists and $j^n(\rank x)=\rank j^n(x)$.
        \item For every $n\in\omega$ and an ordinal $\alpha$, if $j^n(V_\alpha)$ exists, then so does $j^n(\alpha)$ and $j^n(V_\alpha) = V_{j^n(\alpha)}$.
        \item For every $n\in\omega$ and $x$, if $j^n(x)$ exists, then there is an ordinal $\alpha$ such that $x\in V_\alpha$ and $j^n(V_\alpha)$ exists.
        \qedhere 
    \end{enumerate}
\end{proposition}

If we further assume $\Sigma^j_1$-Induction, we can see that iterating $j$ defines a total function:
\begin{proposition}[{\cite[Proposition 4.4]{Corazza2006}}] \pushQED{\qed} \label{Proposition: Sect2-ConsequencesSigmaj1Ind}
    $\ZFC+\BTEE$ with $\Sigma^j_1$-Induction proves the following:
    \begin{enumerate}
        \item For given $n<\omega$ and $x$, there is a unique $y$ such that $\jiter(n,x,y)$. In other words, $j^n(x)$ exists for all $n<\omega$.
        \item For given $n<\omega$, there is a unique ordinal $y$ such that $\critseq(n,y)$. In other words, $j^n(\kappa)$ exists for all $n<\omega$.
        \qedhere 
    \end{enumerate}
\end{proposition}

We can see that $\Sigma^j_n$-Induction is equivalent to $\Pi^j_n$-Induction.
The author included a proof for the implication by modifying a proof of \cite[Corollary 5.4]{Corazza2006} based on Don Hatch's unpublished result. 
However, the reviewer pointed out that the following argument is folklore known in the 1970s at the latest. \cite[\S1(D)]{Parsons1972nQInd}
\begin{proposition}[$\mathsf{Z}^-_0$]
    $\Sigma^j_n$-Induction implies $\Pi^j_n$-Induction, and vice versa.
\end{proposition}
\begin{proof}
    We only consider the implication from $\Sigma^j_n$-Induction to $\Pi^j_n$-Induction since the remaining case is analogous.
    The case $n=0$ is trivial, so let us consider the case $n>0$. Assume the contrary, suppose that $\phi(n,p)$ is a $\Pi^j_n$-formula that satisfy $\phi(0,p)$ and $\forall n<\omega [(\phi(n,p)\to\phi(n+1,p)]$, but $\forall n<\omega\phi(n,p)$ fails. Thus there is $N\in\omega$ such that $\lnot\phi(N,p)$.

    Now consider the formula $\psi(n,p,N)$ given by
    \begin{equation*}
        \psi(n,p,N)\equiv n\le N \to \lnot\phi(N-n,p).
    \end{equation*}
    $\psi(n,p,N)$ is equivalent to a $\Sigma^j_n$-formula. We claim that induction on $n$ for $\psi$ fails. Assume the contrary that the induction for $\psi$ holds. Clearly $\psi(0,p,N)$ holds. If $\psi(n,p,N)$ holds and $n+1\le N$, then we have $\psi(n+1,p,N)$ from $\phi(N-n-1,p)\to\phi(N-n,p)$.
    Hence $\psi(n,p,N)$ holds for all $n$, and especially, $\psi(N,p,N)$ holds, which is equivalent to $\lnot\phi(0,p)$, a contradiction.
\end{proof}

    Finally, let us define the Wholeness axiom as the combination of $\BTEE$ and $\Sigma^j_n$-Separation. We know that $\Sigma^j_n$-Separation implies $\Sigma^j_n$-Induction, so we can use consequences of $\Sigma^j_n$-Induction when we assume $\WA_n$.
Also, we have the following results:
\begin{proposition} \pushQED{\qed}
    $\ZFC+\BTEE$ shows $\Sigma^j_0$-Separation is equivalent to the claim that $j\restricts a$ is a set for any set $a$. \qedhere 
\end{proposition}

\begin{proposition} \pushQED{\qed} \label{Proposition: Sect2-ConsequenceWA0}
    $\ZFC+\WA_0$ proves the following: 
    \begin{enumerate}
        \item (\cite[Proposition 8.4]{Corazza2006}) For any ordinal $\alpha$, we can find $n<\omega$ such that $\alpha\le j^n(\kappa)$, where $\kappa$ is a critical point of $j$.
        \item (\cite[Lemma 8.6.(2)]{Corazza2006}) For every $M$, if $j^n(M)$ exists, then $j^n\restricts M\colon M\to j^n(M)$ is an elementary embedding.
        \item (\cite[Lemma 8.6.(3)]{Corazza2006}) For every $n\ge 1$, $j^n\restricts V_\kappa$ is the identity map.
        \item (\cite[Theorem 8.8]{Corazza2006}) For each $n<\omega$, if $j^n(\kappa)$ exists, then $V_\kappa\prec V_{j(\kappa)}\prec\cdots \prec V_{j^n(\kappa)}$. Also, under the same assumption, $V_{j^n(\kappa)}\prec V$.
        \qedhere 
    \end{enumerate}
\end{proposition}
We should be careful when we talk about $V_{j^n(\kappa)}\prec V$. This statement is actually a schema on a formula, and what it says is for each formula $\phi(x)$ and $a\in V_{j^n(\kappa)}$, we have $\phi(a)\lr V_{j^n(\kappa)}\vDash\phi(a)$. 
Finally, we define $\WA$ as the combination of $\BTEE$ and $\Sigma^j_\infty$-Separation, that is, Separation for arbitrary $j$-formulas.

\begin{Fleischmann}
\subsection{Variations of $\WA$} In this subsection, we define variants of $\WA_n$ based on alternative collection of formulas. 
In the previous subsection on hierarchies of formulas, we defined $\uSigma^j_n$-formulas. As we defined $\WA_n$ as a combination of $\BTEE$ and $\Sigma^j_n$-Separation, let us define $\uWA_n$ as a combination of $\BTEE$ and $\uSigma^j_n$-formulas. It is clear that $\uSigma^j_n$-Separation implies $\uSigma^j_n$-Induction. Moreover, by following the standard proof of Induction on $\omega$, we have the following:
\begin{lemma}\pushQED{\qed}
    Let $n\ge 1$. Then $\uSigma^j_n$-Induction proves $\uSigma^j_1$-Recursion: For a $\uSigma^j_n$-definable class function $F\colon V\to V$, we can find a $\uSigma^j_b$-definable class function $G\colon \omega\to V$ such that $G\restricts n$ is a set for each $n<\omega$ and $G(n)=F(G\restricts n)$.
    \qedhere 
\end{lemma}
\end{Fleischmann}

\subsection{Reducing axioms}
Before finishing this section, let us introduce a form of $\Sigma_n^j$-Separation that uses a minimal number of quantifiers.

\begin{definition}\label{Definition: ReducedAxioms}
    \emph{Reduced $\Gamma$-Separation scheme} is the statement of the form $\forall a\exists b \Sep'_\phi(a,b)$ for a $\Gamma$-formula $\phi(u)$ with a only single free variable $u$, where $\Sep'_\phi(a,b)$ is the following formula:
    \begin{equation*}\label{Formula: Sect3-InstanceSimpleWholeness}
        \forall x\in a(\phi(x)\to x\in b) \land \forall x\in b (x\in a\land \phi(x)).
    \end{equation*}
\end{definition}

\begin{lemma}[$\mathsf{Z}_0^-+\Delta_0$-Replacement]\label{Lemma: SimplifySeparation}
    Let $\Gamma$ be one of either $\Delta^j_0$, $\Sigma^j_n$, or $\Pi^j_n$.
    Then the reduced $\Gamma$-Separation proves the usual $\Gamma$-Separation.
\end{lemma}
\begin{proof}
    Let $a$ be a non-empty set and $\phi(x,y)$ be a $\Gamma$-formula. We want to show that for a parameter set $p$, there is a set $b$ such that
    \begin{equation*}
        \forall x[x\in b \lr x\in a\land \phi(x,p)].
    \end{equation*}
    Let $a'=\{\langle x,p\rangle\mid x\in a\}$.
    $a'$ exists since the map $x\mapsto \langle x,p\rangle$ is a $\Delta_0$-definable class function, so we can apply the $\Delta_0$-replacement to get $a'$. By the reduced $\Gamma$-Separation, the following set exists:
    \begin{equation*}
        b' = \{x\in a' \mid \phi(\pi_0(x),\pi_1(x)) \}
    \end{equation*}
    where $\pi_0$ and $\pi_1$ are projection functions for the Kuratowski ordered pair, which are $\Delta_0$-definable.
    We can check by induction on a complexity of $\phi$ that $\phi(\pi_0(x),\pi_1(x))$ is a $\Gamma$-formula if $\phi(x,y)$ is a $\Gamma$-formula, so we can apply the reduced $\Gamma$-Separation to get $b'$.
    Then $b=\pi_0^"[b'] = \{x \mid \langle x,p\rangle\in b'\}$, and we can see that $b$ satisfies $b=\{x\in a\mid\phi(x,p)\}$.
\end{proof}

The following theorem says we can replace Critical Point of $\BTEE$ to a simpler axiom under the Elementarity of $j$ and $\Delta^j_0$-Separation:
\begin{proposition}[{\cite[p337]{Corazza2000}}, $\ZFC+\Delta^j_0\text{-Separation + Elementarity}$] \pushQED{\qed}
    Critical Point follows from the \emph{Nontriviality of $j$} asserting  $\exists x [x\neq j(x)]$. \qedhere 
\end{proposition}

\section{Partial truth predicates}
\autoref{Proposition: Sect2-ConsequenceWA0} suggests that $\WA_0$ allows us to define the full truth predicate for $\Sigma_\infty$-formulas by using $V_{j^m(\kappa)}$ of parameters with small rank. This truth predicate is, however, not enough to evaluate formulas with $j$. Thus we can ask if there is a way to push this construction to get a partial truth predicate for formulas using $j$.
In this section, as an affirmative answer to the previous question, we see how to construct a truth predicate for $\Delta^j_0(\Sigma_\infty)$-formulas over $\ZFC + \WA_0$ + $\Pi^j_1$-Induction, hence also over $\ZFC+\WA_1$.

Before defining the truth predicate over $\ZFC + \WA_0$ + $\Pi^j_1$-Induction, let us define subsidiary terminologies for $j$-formulas. For a $j$-formula $\phi$, a \emph{$j$-height of $\phi$} is the number of $j$ appearing in $\phi$. For a set $a$, a \emph{degree} of $a$ is the least $n<\omega$ such that $a\in V_{j^n(\kappa)}$.
For $\vec{a}$ a finite sequence of sets, the degree of $\vec{a}$ means the maximum of the degree of components of $\vec{a}$. We denote the $j$-height of $\phi$ and the degree of $a$ as $\jht\phi$ and $\deg a$ respectively.  To clarify the meaning of the definition, let us provide their definitions in a formal manner:
\begin{definition}
    For a formula $\phi$, let us define its \emph{$j$-height} $\jht\phi$ inductively as follows:
    \begin{itemize}
        \item If $\phi(v_0,v_1,\cdots,v_n)$ is $\Sigma_\infty$, the for not necessarily distinct variables $x_0$, $x_1$, $\cdots$, $x_n$, define $\jht \phi(j^{e_0}(x_0),\cdots, j^{e_n}(x_n)) = e_0+\cdots +e_n$.
        \item $\jht (\lnot\phi)=\jht\phi$
        \item For a propositional connective $\circ\in \{\land,\lor,\to\}$, $\jht(\phi\circ\psi)=\jht\phi + \jht\psi$.
        \item If $\mathsf{Q}$ is a quantifier, define $\jht (\mathsf{Q}x \phi) = \jht \phi$ and $\jht(\mathsf{Q}x\in j^l(y)\phi) = l+\jht(\phi)$.
    \end{itemize}
    
    For a set $a$, define $\deg_1 a = \min\{n<\omega \mid \rank a < j^n(\kappa)\}$. For a finite tuple of sets $\vec{a}=\langle a_0,\cdots, a_{m-1}\rangle$, $\deg\vec{a}$ is the maximum of $\deg_1(a_i)$ for $i<m$.
\end{definition}
The definition of $\jht$ is coherent for quantifiers in the sense that
\begin{equation*}
    \jht(\forall x( x\in j^l(y)\to \phi)) =
    \jht(\exists  x( x\in j^l(y)\land  \phi)) = l + \jht(\phi).
\end{equation*}
$\jht$ is recursively defined and takes a natural number (coding a formula). Thus $\jht$ is a recursive function. 
The definition of degree function requires some justification:
\begin{lemma}[$\ZFC+\WA_0$ + $\Pi^j_1$-Induction] \label{Lemma: degree is Delta j 1 definable}
    $\deg_1$ and $\deg$ are $\Delta^j_1$-definable total class functions with parameter $V_\kappa$.
\end{lemma}
\begin{proof}
    We can see that $\deg_1$ satisfies
    \begin{align*}
        \deg_1(x)=n & \iff \forall f \big[\prejiter(f,n,V_\kappa) \to \big( a\in f(n) \land \forall m<n [a \notin f(m)]\big)\big] \\
        & \iff  \exists f \big[\prejiter(f,n,V_\kappa) \land \big( a \in f(n) \land \forall m<n [a \notin f(m)]\big)\big]
    \end{align*}
    $\prejiter(f,n,\kappa)$ is $\Sigma^j_1$, so $\deg_1(x)$ is $\Delta^j_1$-definable with parameter $V_\kappa$.
    To formulate the definition of $\deg$, we appeal to \autoref{Corollary: Sect2-optimized-recursion} and define $\deg \vec{a}$ by recursion on the length of $\vec{a}$; More precisely, consider the following recursively defined $\Delta^j_1$-function $G(n,\vec{a})$ with parameter $\vec{a}$:
    \begin{equation*}
        G(0,\vec{a})=0,\ G(n+1,\vec{a}) = \max\{G(n,\vec{a}),\vec{a}(n)\}.
    \end{equation*}
    (We take $\vec{a}(i)=0$ if $i$ is greater than or equal to the length of $\vec{a}_i$.) Then $G(n,\vec{a})$ is $\Delta^j_1$-definable, and $\deg(\vec{a}) = G(\dom (\vec{a}),\vec{a})$.
\end{proof}

The reader may ask why we defined the $j$-height of a formula. \autoref{Proposition: Sect2-Reflection0} ensures how much `additional height' is required to catch the validity of a $\Delta^j_0(\Sigma_\infty)$-formula $\phi$. Before providing its formal statement, let us examine its brief motivation.

If $\phi$ is $\Sigma_\infty$, then its truth is captured by any $V_{j^n(\kappa)}$ for any $n$ such that parameters of $\phi$ is in $V_{j^n(\kappa)}$. Especially, if $\phi(\vec{a})$ holds and $m\ge\deg(\vec{a})$, then $V_{j^m(\kappa)}\vDash \phi(\vec{a})$ also holds.
We want to have a similar situation for $\Delta_0^j(\Sigma_\infty)$-formulas, but the situation is more cumbersome when $\phi$ is a proper $\Delta_0^j(\Sigma_\infty)$-formula. For example, for a $\Sigma_\infty$-formula $\psi(x)$, consider the $\Delta_0^j(\Sigma_\infty)$-formula $\phi(a)\equiv \exists x\in a \psi(j(x))$.
Even when $a\in V_{j^n(\kappa)}$, there is no guarantee that $j(x)\in V_{j^n(\kappa)}$, so even stating $V_{j^n(\kappa)}\vDash \phi(a)$ may not make sense.
However, if $a\in V_{j^n(\kappa)}$ and $x\in a$, then $j(x)\in V_{j^{n+1}(\kappa)}$. Thus $V_{j^{n+1}(\kappa)}$ is `enough' to catch the formula $\phi(a)$, and the additional $+1$ is equal to $\jht(\phi)$. 

\begin{proposition}[$\ZFC+\WA_0$ + $\Pi^j_1$-Induction]\label{Proposition: Sect2-Reflection0}
    Let $\phi(\vec{x})$ be a $\Delta^j_0(\Sigma_\infty)$ formula. If $\vec{a}$ is a  finite sequence of sets, then for each natural number $m\ge \deg \vec{a} + \jht\phi$, we have
    \begin{equation*}
        \phi(\vec{a})\iff (V_{j^m(\kappa)},\in, \hat{\jmath}_m)\vDash\phi(\vec{a}).
    \end{equation*}
    Here $\hat{\jmath}_m$ is a function of domain $V_{j^m(\kappa)}$ given by
    \begin{equation*}
        \hat{\jmath}_m(x) = 
        \begin{cases}
            j(x) & j(x)\in V_{j^m(\kappa)}, \\
            0 & \text{otherwise.}
        \end{cases}
    \end{equation*}
\end{proposition}
\begin{proof}
    Let us prove it by induction on $\phi$:
    \begin{enumerate}
        \item Suppose that $\phi$ takes the form $\psi(j^{e_0}(x_0),\cdots,j^{e_n}(x_n))$ for some $\Sigma_\infty$-formula $\psi$ and not necessarily pairwise distinct variables $x_0$, $\cdots$, $x_n$.
        If $\vec{a}=\langle a_0,\cdots, a_{n-1}\rangle$, we can see that $m\ge \deg \vec{a} + \jht\phi$ and $j^{e_i}(a_i)\in V_{j^m(\kappa)}$ for $i=0,1,\cdots, n$. Thus $j^{e_i}(a_i) = \hat{\jmath}_m^{e_i}(a_i)$ for each $i$, and so the desired result follows from \autoref{Proposition: Sect2-ConsequenceWA0}.

        \item Now assume that the proposition holds for $\phi$ and $\psi$. Then it is easy to see that for $m\ge \deg \vec{a}+ \jht\phi + \jht\psi$, $\phi\circ\psi(\vec{a})$ if and only if the same holds over $(V_{j^m(\vec{a})},\in, \hat{\jmath}_m)$. The case for the negation is easy to see.

        \item Now let $\phi$ is of the form $\exists x\in j^e(y) \psi(x,y,\vec{z})$. We have
        \begin{equation*}
            \jht(\exists x\in j^e(y) \psi(x,y,\vec{z}))=e + \jht \psi(x,y,\vec{z}).
        \end{equation*}
        
        Now suppose that $m\ge \deg{\vec{a}}+ e + \jht\psi$, and let $\vec{a}_y$ be the component of $\vec{a}$ corresponding to the variable $y$.
        Observe that $\hat{\jmath}_m^e(\vec{a}_y) = j^e(\vec{a}_y)$ since $e+\deg(\vec{a})\le \jht(\phi)+\deg(\vec{a})\le m$.
        
        For one direction, suppose that there is $b\in j^e(\vec{a}_y)$ such that $\psi(b,\vec{a}_y,\vec{z})$.
        Then $\deg b \le e + \deg\vec{a}$, so 
        \begin{equation*}
            \deg(\vec{a}[b/x]) + \jht(\psi) \le \deg\vec{a}+e+\jht\psi\le m.
        \end{equation*}        
        Hence we can apply the inductive hypothesis, and we have 
        \begin{equation*} 
            \psi(b,\vec{a}_y,\vec{a}_{\vec{z}}) \iff \left(V_{j^m(\kappa)},\in, \hat{\jmath}_m\right)\vDash \psi(b,\vec{a}_y,\vec{a}_{\vec{z}}).
        \end{equation*}
        Hence we have $\left(V_{j^m(\kappa)},\in, \hat{\jmath}_m\right)\vDash \exists x\in j^e(\vec{a}_y) \psi(x,\vec{a}_y,\vec{a}_{\vec{z}})$.
        For the remaining direction, suppose that we have 
        \begin{equation*}
            \left(V_{j^m(\kappa)},\in,\hat{\jmath}_m\right)\vDash \exists x\in j^e(\vec{a}_y) \psi(x,\vec{a}_y,\vec{a}_{\vec{z}}).
        \end{equation*}
        Then we have $b\in \hat{\jmath}_m^e(\vec{a}_y)$ such that $\left(V_{j^m(\kappa)},\in, \hat{\jmath}_m\right)\vDash \psi(b,\vec{a}_y,\vec{a}_{\vec{z}})$.
        Now we can apply the inductive hypothesis for $\psi$ since 
        \begin{equation*}
            \deg(b,\vec{a}_y,\vec{a}_{\vec{z}})+\jht(\psi) \le \max(e,\deg(\vec{a}))+\jht(\psi) \le m,
        \end{equation*}
        and so we have $\psi(b,\vec{a}_y,\vec{a}_{\vec{z}})$ as desired.
        
        \item Lastly, the case when $\phi$ is of the form $\forall x\in j^e(y) \psi(x,y,\vec{z})$ is similar to the existential case, so we left it to the reader.
        \qedhere 
    \end{enumerate}
\end{proof}

Note that the above proposition is actually a proposition schema, so we formally have instances of \autoref{Proposition: Sect2-Reflection0} for each $\Delta^j_0(\Sigma_\infty)$ formula $\phi$, and induction on $\phi$ in the proof applies over the metatheory
We compute $\jht\phi$ as the $j$-height of a coded formula corresponding to $\phi$, and the resulting $j$-height is a standard natural number. However, $m$ and $\deg\vec{a}$ may not be standard.

Now we want to define a partial truth predicate $\vDash_{\Delta^j_0(\Sigma_\infty)}\phi[\vec{a}]$ for a (code for) $\Delta^j_0(\Sigma_\infty)$-formula $\phi$ and a sequence of sets $\vec{a}$ as $(V_{j^m(\kappa)},\in, \hat{\jmath}_m)\vDash \phi(\vec{a})$ for an appropriate $m\ge \deg\vec{a} + \jht(\phi)$.
We prove that the choice of $m$ is irrelevant as long as $m\ge \deg\vec{a} + \jht(\phi)$ to ensure the definition works:
\begin{lemma}[$\ZFC + \WA_0$ + $\Pi^j_1$-Induction] \label{Lemma: Sect3-Eventual Delta j 0 Sigma infty truth coherence}
    Suppose that $\phi$ is a (code for) $\Delta^j_0(\Sigma_\infty)$-formula and $\vec{a}$ is a finite sequence of sets.
    Then
    \begin{equation} \label{Formula: Equivalence of Delta j 0 Sigma infty truth}
        \forall m,m'<\omega \forall \vec{a} \in V^{<\omega}
        \left[ m,m' \ge \deg\vec{a} + \jht(\phi)\implies  \left((V_{j^m(\kappa)},\in, \hat{\jmath}_m)\vDash\phi(\vec{a}) \iff (V_{j^{m'}(\kappa)},\in, \hat{\jmath}_{m'})\vDash\phi(\vec{a})\right)\right].
    \end{equation}
\end{lemma}
\begin{proof}
    Let us fix a recursive enumeration of formulas in the language of $\{\in,j\}$ such that every subformula of a given formula occurs earlier in the enumeration than the formula. 
    We first claim that \eqref{Formula: Equivalence of Delta j 0 Sigma infty truth} is equivalent to a $\Pi^j_1$-formula with a natural number free variable (a code for) $\phi$ and parameter $V_\kappa$, so we can apply $\Pi^j_1$-induction.
    Let
    \begin{multline*}
        \Xi(v,f,m) \equiv \jiter(m,V_\kappa,v)\land \text{``$f$ is a function''}\land \dom f = v \\ \land \forall x\in v \big[\big(j(x)\in v\to f(x)=j(x)\big) \land (j(x)\notin v\to f(x)=0)\big]
    \end{multline*}
    Then $\Xi(v,f,m)$ is $\Sigma^j_1$-formula with parameter $V_\kappa$. We can also see $\forall m<\omega\exists! v,f\ \Xi(v,f,m)$ from \autoref{Proposition: Sect2-ConsequencesSigmaj0Ind}. (We need $\Delta^j_0$-Separation to show the existence of $\hat{\jmath}_m$.)
    Moreover, $\Xi(v,f,m)$ implies $v=V_{j^m(\kappa)}$ and $f=\hat{\jmath}_m$.
    Hence $(V_{j^m(\kappa)},\in, \hat{\jmath}_m)\vDash\phi(\vec{a})$ has the following two formulations:
    \begin{enumerate}
        \item $\forall v\forall f [\Xi(v,f,m)\to (v,\in,f)\vDash \phi(\vec{a})]$, and 
        \item $\exists v\forall f [\Xi(v,f,m)\land  (v,\in,f)\vDash \phi(\vec{a})]$.
    \end{enumerate}
    Hence $(V_{j^m(\kappa)},\in, \hat{\jmath}_m)\vDash\phi(\vec{a})$ is $\Delta^j_1$ with parameter $V_\kappa,m,\phi,\vec{a}$.
    The expression $\vec{a}\in V^{<\omega}$ means $\vec{a}$ is a finite function of domain $n\in\omega$, which is $\Delta_1$-expressible. $\jht$ is arithmetically definable and $\deg$ is $\Delta^j_1$-definable by \autoref{Lemma: degree is Delta j 1 definable}.
    Hence the entire formula \eqref{Formula: Equivalence of Delta j 0 Sigma infty truth} is $\Pi^j_1$ with parameter $V_\kappa$.

    Now let us prove \eqref{Formula: Equivalence of Delta j 0 Sigma infty truth} by $\Pi^j_1$-induction applied to a (code for a) formula $\phi$. 
    \begin{enumerate}
        \item $\phi$ is atomic: We may assume $\phi(v_0,v_1)$ is either $j^{e_0}(v_0)=j^{e_1}(v_1)$ or $j^{e_0}(v_0)\in j^{e_1}(v_1)$. In either case, the interpretation of $j$ (as a symbol) makes the proof non-trivial.
        For $\vec{a}=\langle a_0,a_1,\cdots\rangle$, it suffices to show that $\hat{\jmath}^{e_i}_m(a_i) = j^{e_i}(a_i)=\hat{\jmath}^{e_i}_{m'}(a_i)$ for $i=0,1$, which follows from
        \begin{equation*}
            \deg_1(a_i) + e_i \le \deg(\vec{a}) + \jht(\phi) \le m,m'.
        \end{equation*}

        \item The cases for connectives are easy to prove.

        \item Now suppose that $\phi$ is $\exists x\in j^e(y) \psi(x,y,\vec{z})$. By mimicking the argument in the proof of \autoref{Proposition: Sect2-Reflection0}, we have that for every $m\ge\deg(\vec{a})+\jht(\phi)$,
        \begin{equation*}
            (V_{j^m(\kappa)},\in, \hat{\jmath}_m)\vDash\phi(\vec{a}) \iff \exists b\in j^e(\vec{a}_y) (V_{j^m(\kappa)},\in, \hat{\jmath}_m)\vDash\psi(\vec{a}[b/y]).
        \end{equation*}
        By the inductive hypothesis, we have for every $b$ and $m,m'\ge \deg(\vec{a}[b/y])+\jht(\phi)$,
        \begin{equation*}
            (V_{j^m(\kappa)},\in, \hat{\jmath}_m)\vDash\psi(\vec{a}[b/y]) \iff (V_{j^{m'}(\kappa)},\in, \hat{\jmath}_{m'})\vDash\psi(\vec{a}[b/y]).
        \end{equation*}
        Hence the desired equivalence follows.

        \item The case for bounded universal quantification is similar, so we leave it to the reader. \qedhere 
    \end{enumerate}
\end{proof}

\autoref{Proposition: Sect2-Reflection0} and \autoref{Lemma: Sect3-Eventual Delta j 0 Sigma infty truth coherence} enable us to define the truth predicate for $\Delta^j_0(\Sigma_\infty)$-formulas:
\begin{definition}
    Define a partial truth predicate $\vDash_{\Delta^j_0(\Sigma_\infty)}$ for $\Delta^j_0(\Sigma_\infty)$-formulas as follows: 
    Let $\phi$ be a (code for) $\Delta^j_0(\Sigma_\infty)$-formula. Then we say $\vDash_{\Delta^j_0(\Sigma_\infty)}\phi(\vec{a})$ if and only if for every $m\ge \deg\vec{a}+\jht \phi$, $(V_{j^m(\kappa)},\in, \hat{\jmath}_m)\vDash \phi(\vec{a})$.
\end{definition}

We can see that $\vDash_{\Delta^j_0(\Sigma_\infty)}$ satisfies Tarski's definition of truth. In particular, we have
\begin{lemma}[$\ZFC + \WA_0$ + $\Pi^j_1$-Induction] \pushQED{\qed}
    For two $\Delta^j_0(\Sigma_\infty)$-formulas $\phi$ and $\psi$, we have
    \begin{enumerate}
        \item $\vDash_{\Delta^j_0(\Sigma_\infty)} \lnot \phi(\vec{a}) \iff \nvDash_{\Delta^j_0(\Sigma_\infty)}\phi(\vec{a})$,
        \item $\vDash_{\Delta^j_0(\Sigma_\infty)} \phi\circ \psi(\vec{a}) \iff \big[ \vDash_{\Delta^j_0(\Sigma_\infty)} \phi(\vec{a})\big] \circ \big[ \vDash_{\Delta^j_0(\Sigma_\infty)} \phi(\vec{a})\big]$ for a binary connective $\circ$. \qedhere 
    \end{enumerate}
\end{lemma}

Moreover, $\vDash_{\Delta^j_0(\Sigma_\infty)}$ calculates $j$ correctly in the following sense:
\begin{lemma}[$\ZFC + \WA_0$ + $\Pi^j_1$-Induction]
    Let $\phi$ be a $\Delta^j_0(\Sigma_\infty)$-formula.
    Then 
    \begin{equation*}
        \vDash_{\Delta^j_0(\Sigma_\infty)}(\phi(j^{e_0}(v_0),\cdots, j^{e_{k-1}}(v_{k-1})))[\langle a_0,\cdots, a_{k-1}\rangle] \iff \vDash_{\Delta^j_0(\Sigma_\infty)}\phi[\langle j^{e_0}(a_0),\cdots,j^{e_{k-1}}(a_{k-1})\rangle]
    \end{equation*}
    for every $a_0,\cdots, a_{k-1}$.
\end{lemma}
\begin{proof}
    Let $\vec{a}=\langle a_0,\cdots,a_{k-1}\rangle$.
    First, 
    \begin{equation} \label{Formula: Sec3-Models-phi-j-00}
        \vDash_{\Delta^j_0(\Sigma_\infty)}(\phi(j^{e_0}(v_0),\cdots, j^{e_{k-1}}(v_{k-1})))[\langle a_0,\cdots, a_{k-1}\rangle]
    \end{equation}
    if and only if for every $m\ge \deg(\vec{a}) + \jht(\phi)+e_0+\cdots+e_{k-1}$, $\left( V_{j^m(\kappa)},\in,\hat{\jmath}_m\right)\vDash \phi(\hat{\jmath}_m^{e_0}(a_0),\cdots,\hat{\jmath}_m^{e_{k-1}}(a_{k-1}))$.
    Since $m\ge \deg_1(a_i)+e_i$ for each $i<k$, we have that $j^{e_i}(a_i) \in V_{\kappa_m}$, so $\hat{\jmath}_m^{e_i}(a_i) = j^{e_i}(a_i)$ for each $i<k$.
    Hence \eqref{Formula: Sec3-Models-phi-j-00} if and only if for every $m\ge \deg(\vec{a}) + \jht(\phi)+e_0+\cdots+e_{k-1}$, $V_{j^m(\kappa)}\vDash \phi(j^{e_0}(a_0),\cdots, j^{e_{k-1}}(a_{k-1}))$.
    Similarly, we can see that 
    \begin{equation*}
        \vDash_{\Delta^j_0(\Sigma_\infty)}\phi[\langle j^{e_0}(a_0),\cdots,j^{e_{k-1}}(a_{k-1})\rangle]
    \end{equation*}
    if and only if for every $m\ge \jht(\phi)+\max\{e_i+\deg_1(a_i)\mid i<k\}$, $V_{j^m(\kappa)}\vDash \phi(j^{e_0}(a_0),\cdots, j^{e_{k-1}}(a_{k-1}))$.
    So in both cases, each formula holds when for a sufficiently large $m$, $V_{j^m(\kappa)}\vDash \phi(j^{e_0}(a_0),\cdots, j^{e_{k-1}}(a_{k-1}))$. Hence the equivalence follows by \autoref{Lemma: Sect3-Eventual Delta j 0 Sigma infty truth coherence}.
\end{proof}

Let us analyze the definitional complexity of $\vDash_{\Delta^j_0(\Sigma_\infty)}$. For a given (code for) $\Delta^j_0(\Sigma_\infty)$-formula $\phi$ and a finite sequence of sets $\vec{a}$, 
\begin{equation*}
    \vDash_{\Delta^j_0(\Sigma_\infty)} \phi(\vec{a})\iff \forall m<\omega \big[m\ge(\deg \vec{a} +\jht\phi) \to (V_{j^m(\kappa)},\in, \hat{\jmath}_m)\vDash \phi(\vec{a})\big].
\end{equation*}

We can see that the inequality $m\ge \deg\vec{a}+\jht \phi$ is $\Delta^j_1$ since $\deg$ is $\Delta^j_1$-definable.
Also, $(V_{j^m(\kappa)},\in, \hat{\jmath}_m)\vDash \phi(\vec{a})$ is $\Delta^j_1$ as we proved in the proof of \autoref{Lemma: Sect3-Eventual Delta j 0 Sigma infty truth coherence}.
Hence $\vDash_{\Delta^j_0(\Sigma_\infty)}$ is $\Pi^j_1$-definable. Furthermore, we also have
\begin{equation}\label{Formula: Sect3-Sigmaj1-formula-Pred0}
    \vDash_{\Delta^j_0(\Sigma_\infty)} \phi(\vec{a})\iff \exists m<\omega \big[m\ge(\deg \vec{a} +\jht\phi) \land (V_{j^m(\kappa)},\in, \hat{\jmath}_m)\vDash \phi(\vec{a})\big]
\end{equation}
Hence $\vDash_{\Delta^j_0(\Sigma_\infty)}$ is $\Delta^j_1(\Sigma_\infty)$-definable provably in $\ZFC + \WA_0$ + $\Pi^j_1$-Induction.

We can also define a partial truth predicate for $\Sigma^j_n(\Sigma_\infty)$-formulas whose definitional complexity is $\Sigma^j_n$.
\begin{definition}
    Let us define $\vDash_{\Sigma^j_n(\Sigma_\infty)}$ inductively as follows: if $\phi(x,y)$ is a $\Pi^j_{n-1}(\Sigma_\infty)$-formula, then define $\vDash_{\Sigma^j_n(\Sigma_\infty)} \forall x \phi(x,p)$ if and only if there is $a$ such that $\nvDash_{\Sigma^j_{n-1}(\Sigma_\infty)}\lnot\phi(a,p)$. 
\end{definition}

The above truth predicate is clearly $\Sigma^j_n$-definable.
We can also define $\vDash_{\Pi^j_n(\Sigma_\infty)}\phi(x)$ when $\phi$ is $\Pi^j_n(\Sigma_\infty)$ by negating $\vDash_{\Sigma^j_n(\Sigma_\infty)}$, and $\vDash_{\Pi^j_n(\Sigma_\infty)}\phi(x)$ is $\Pi^j_n$.
By using $\vDash_{\Sigma^j_n(\Sigma_\infty)}$, we can cast $\Sigma^j_n$-Separation as a single sentence when $n\ge 1$.
Thus if $n\ge 1$, then $\Sigma^j_n$-Separation is finitely axiomatizable.
Combining with the fact in \cite{Hamkins2001WA} that $\ZFC+\WA_0$ is finitely axiomatizable, we have the following:
\begin{corollary}\pushQED{\qed} \label{Corollary: Sect4-WAn finite}
    For each natural number $n$, $\ZFC+\WA_n$ is finitely axiomatizable. \qedhere
\end{corollary}

Before finishing this section, let us expand $\Pi^j_n(\Sigma_\infty)$-truth predicate to combinations of $\Pi^j_n(\Sigma_\infty)$ by conjunction, disjunction, and universal quantification:
\begin{definition}
    \emph{A class of $\underline{\Pi}^j_n(\Sigma_\infty)$-formulas} is the least class of formulas containing $\Pi^j_n(\Sigma_\infty)$ closed under $\land$, $\lor$, and (unbounded) $\forall$.
    We also define  class of $\underline{\Sigma}^j_n(\Sigma_\infty)$-formulas dually, as the least class of formulas containing $\Sigma^j_n(\Sigma_\infty)$ closed under $\land$, $\lor$, and (unbounded) $\exists$.
\end{definition}

By the proof of \autoref{Corollary: Sect2-Effective-normalization}, for a $\underline{\Pi}^j_n(\Sigma_\infty)$ formula $\phi$, we can find a $\Pi^j_n(\Sigma_\infty)$ formula $D_\phi$ with precisely the same free variables with $\phi$ such that $\mathsf{Z}^-_0$ proves $\forall \vec{x}[\phi(\vec{x})\lr D_\phi(\vec{x})]$.
Then let us define $\vDash_{\underline{\Pi}^j_n(\Sigma_\infty)}\phi(\vec{a})$ by $\vDash_{\Pi^j_n(\Sigma_\infty)}D_\phi(\vec{a})$.
We can make $D_\phi$ is identical with $\phi$ when $\phi$ is $\Pi^j_n(\Sigma_\infty)$, so $\vDash_{\underline{\Pi}^j_n(\Sigma_\infty)}\phi(\vec{a})$ iff $\vDash_{\Pi^j_n(\Sigma_\infty)}D_\phi(\vec{a})$ when $\phi$ is $\Pi^j_n(\Sigma_\infty)$.
We can also see that $\vDash_{\underline{\Pi}^j_n(\Sigma_\infty)}$ satisfies Tarskian definition of truth:
\begin{lemma}[$\ZFC+\WA_0$ + $\Pi^j_1$-Induction] \pushQED{\qed}
    For each (standard) $n$ and two (codes for) $\underline{\Pi}^j_n(\Sigma_\infty)$-formulas $\phi$ and $\psi$, we have the following: For every $\vec{a}$,
    \begin{enumerate}
        \item $\vDash_{\underline{\Pi}^j_n(\Sigma_\infty)} (\phi\circ\psi) (\vec{a})$ iff $\big[ \vDash_{\underline{\Pi}^j_n(\Sigma_\infty)} \phi(\vec{a})\big]\circ \big[ \vDash_{\underline{\Pi}^j_n(\Sigma_\infty)} \psi(\vec{a})\big]$ if $\circ \in \{\land,\lor\}$.
        \item $\vDash_{\underline{\Pi}^j_n(\Sigma_\infty)}\forall y \phi(\vec{a})$ iff for every $b$, $\vDash_{\underline{\Pi}^j_n(\Sigma_\infty)}\phi(\vec{a}[b/y])$. \qedhere 
    \end{enumerate}
\end{lemma}

Moreover, we can distribute finitely many disjunctions of $\Pi^j_n(\Sigma_\infty)$-formulas:
\begin{proposition}[$\ZFC + \WA_0$ + $\Pi^j_{n+1}$-Induction]
    Let $\langle\phi_i \mid i<k\rangle$ be a finite enumeration of $\underline{\Pi}^j_n(\Sigma_\infty)$-formulas.
    Then we have the following:
    \begin{equation} \label{Formula: Sect3-Finite Disjunction Distribution}
        \forall \vec{a} \big[ \vDash_{\underline{\Pi}^j_n(\Sigma_\infty)} (\phi_0\lor\cdots\lor \phi_{k-1})(\vec{a}) \iff \exists i<k \big( 
        \vDash_{\underline{\Pi}^j_n(\Sigma_\infty)} \phi_i(\vec{a})\big)\big].
    \end{equation}
\end{proposition}
\begin{proof}
    The proof follows by applying induction on $k\in\omega$ to \eqref{Formula: Sect3-Finite Disjunction Distribution}, which is straightforward. Here we can apply $\Pi^1_{n+1}$-Induction since $\vDash_{\Pi^j_n(\Sigma_\infty)}$ (both for a single $\Pi^j_n(\Sigma_\infty)$-formula or combinations of them) is $\Pi^j_n$-definable, and bounded existential quantifier bounded by a natural number $k$ does not increase the complexity of a $\Pi^j_n$-formula by \autoref{Corollaty: Sect2-BddQreduction}, so \eqref{Formula: Sect3-Finite Disjunction Distribution} is equivalent to a $\Pi^j_{n+1}$-formula.
\end{proof}


\section{Separating Wholeness}
This section proves the consistency implications we announced in the introduction. Unlike most set-theoretic arguments about large cardinals, we will heavily rely on proof-theoretic means.

The author's initial argument constructs a \emph{weak model} to apply the Strong Soundness theorem \cite[Theorem II.8.10]{Simpson2009}. However, the reviewer suggested a more direct proof akin to the proof of completeness theorem and cut-elimination theorem of first-order logic.
All proofs in this section follow the reviewer's suggestion.

\subsection{A proof of $\Con(\ZFC+\WA_0)$}
In this subsection, we prove the consistency of $\ZFC+\WA_0$ from $\ZFC+\WA_0$ + $\Pi^j_1$-Induction, which also implies $\ZFC+\WA_1\vdash \Con(\ZFC+\WA_0)$.

Before starting the proof for the main result, let us review the proof-theoretic tools we will use. First, let us fix the language for the theories we will formulate.
\begin{definition}
    The language we consider is $\{\in,=,\sfj,\dot{\kappa}\}$, where $\dot{\kappa}$ is a constant symbol for a set moved by $j$. (We intend $\dot{\kappa}$ for the critical point of $j$, but we can interpret $\dot{\kappa}$ by any set moved by $j$ in principle.)
    We distinguish bounded quantifiers and unbounded quantifiers as formally distinct objects.
    That is, $\mathsf{Q} x\in y \phi$ is a legitimate formula \emph{only when} $\phi$ is $\Delta^j_0(\Sigma_\infty)$.
\end{definition}

There are different ways to formulate first-order logic in proof theory, and we will use Tait-styled one-sided sequent calculus. Hence there is no negation connective, but we include \emph{literals} (atomic formulas and its negations) as a basic notion. Then we define the negation as follows:
\begin{itemize}
    \item $\lnot(s=t) \equiv (s\neq t)$ and $\lnot(s\neq t) \equiv (s=t)$, and similarly for $\in$ and $\notin$.
    \item $\lnot(\phi\land \psi) \equiv \lnot \phi\lor \lnot\psi$ and $\lnot(\phi\lor \psi) \equiv \lnot \phi\land \lnot\psi$.
    \item $\lnot(\forall x \phi(x))\equiv \exists x \lnot\phi(x)$ and $\lnot(\exists x \phi(x))\equiv \forall x \lnot\phi(x)$, and similarly for bounded quantifiers.
\end{itemize}

Then let us formulate sequent calculus for first-order logic as follows. We follow terminologies and facts in \cite{Arai2020}, but the reader may also refer to \cite[Ch. 4]{Pohlers2009Prooftheory}.
We use $A,B,C,\cdots$ to denote formulas,  $\Gamma,\Delta,\cdots$ to denote sequents, $s, t$ for terms, and $x,y,z,\cdots$ for free variables.
\begin{definition}
    Sequent calculus $\bfG_1$ comprises the following inference rules:
    \begin{center}
    \begin{longtable}{c c}
        \AxiomC{}
        \RightLabel{Ax}
        \UnaryInfC{$\Gamma,\lnot L,L$}
        \DisplayProof
        &
        \\[2em]
        \AxiomC{$\Gamma,A_0\lor A_1,A_i$}
        \RightLabel{$\lor_i$}
        \UnaryInfC{$\Gamma, A_0\lor A_1$}
        \DisplayProof
        &
        \AxiomC{$\Gamma, A_0\land A_1, A_0$}
        \AxiomC{$\Gamma, A_0\land A_1, A_1$}
        \RightLabel{$\land$}
        \BinaryInfC{$\Gamma, A_0\land A_1$}
        \DisplayProof
        \\[2em]
        \AxiomC{$\Gamma,\exists x A(x),A(t)$}
        \RightLabel{$\exists$}
        \UnaryInfC{$\Gamma,\exists x A(x)$}
        \DisplayProof
        &
        \AxiomC{$\Gamma,\forall x A(x),A(y)$}
        \RightLabel{$\forall$}
        \UnaryInfC{$\Gamma,\forall x A(x)$}
        \DisplayProof
    \end{longtable}
    \end{center}
    Here Ax applies only when $L$ is a literal, and $\forall$ applies only when $y$ does not occur in the lower sequent $\Gamma,\forall x A(x)$ freely. We call $y$ in the $\forall$ rule an \emph{eigenvariable of the inference rule $\forall$}.
\end{definition}
Note that the \emph{cut rule}
\begin{center}
    \AxiomC{$\Gamma, \lnot A$}
    \AxiomC{$A, \Delta$}
    \RightLabel{$\mathrm{Cut}$}
    \BinaryInfC{$\Gamma, \Delta$}
    \DisplayProof
\end{center}
is not part of $\bfG_1$. It is known that $\bfG_1$ with the Cut rule is sound \cite[Proposition 1.1]{Arai2020}, and $\bfG_1$ (even without the Cut rule) is complete:
\begin{theorem}[Completeness theorem, {\cite[Theorem 2.1]{Arai2020}, \cite[Theorem 4.4.1]{Pohlers2009Prooftheory}}]
    If $\bigvee \Gamma$ is valid, then $\bfG_1\vdash \Gamma$.
\end{theorem}

The completeness theorem holds for not only $\bfG_1$ but also a modification of $\bfG_1$ augmented with new inference rules associated with a given theory $T$.
The easiest way to add a new inference associated with $A\in T$ is by adding the rule
\begin{center}
    \AxiomC{$\Gamma,\lnot A$}
    \RightLabel{$\mathrm{Ax}_A$}
    \UnaryInfC{$\Gamma$}
    \DisplayProof
\end{center}
When $A$ takes the form $\forall\vec{y} \left[ \left[\bigwedge_{i<n} \forall \vec{x}_i \bigvee \Phi_i(\vec{x}_i,\vec{y}) \right]\to \bigwedge \Phi(\vec{y})\right]$, we may add the following rule instead:
\begin{center}
    \AxiomC{$\{\Gamma,\Phi_i(\vec{z}_i, \vec{t}) \mid i<n\}$}
    \RightLabel{$\mathrm{Ax}'_A$}
    \UnaryInfC{$\Gamma,\Phi(\vec{t})$}
    \DisplayProof
\end{center}
where $\vec{z}_i$ do not occur free in $\Gamma,\Phi(\vec{t})$. Let us say $\bfG_T \supseteq \bfG_1$ is an \emph{admissible expansion for $T$} if $\bfG_T$ comprises rules of $\bfG_1$ and one of $\mathrm{Ax}_A$ or $\mathrm{Ax}'_A$ for each $A\in T$.
The following theorem follows a proof of \cite[Theorem 2.2]{Arai2020}:
\begin{theorem} \pushQED{\qed} \label{Theorem: Sect4-Expanded cutfree completeness}
    If $T\vDash \bigvee\Gamma$, then $\bfG_T\vdash \Gamma$ for an admissible expansion $\bfG_T$ for $T$. \qedhere 
\end{theorem}

Now let us consider the case $T=\ZFC+\WA_0$.
We consider the following admissible expansion $\bfG_{\WA_0}$: It comprises axioms of $\bfG_1$ and
\begin{center}
\begin{longtable}{c c}
    \AxiomC{$\Gamma,\exists x\in t A(x),s\in t$}
    \AxiomC{$\Gamma,\exists x\in t A(x),A(s)$}
    \RightLabel{$b\exists$}
    \BinaryInfC{$\Gamma,\exists x\in t A(x)$}
    \DisplayProof
    &
    \AxiomC{$\Gamma,y\notin t,A(y)$}
    \RightLabel{$b\forall$ ($y$ free in the lower sequent.)}
    \UnaryInfC{$\Gamma,\forall x\in t A(x)$}
    \DisplayProof
    \\[2em]
    \AxiomC{$\Gamma,t\neq t$}
    \RightLabel{Eq-refl}
    \UnaryInfC{$\Gamma$}
    \DisplayProof
    &
    \AxiomC{$\Gamma,t_0=t_1$}
    \AxiomC{$\Gamma,t_1=t_2$}
    \AxiomC{$\Gamma,t_0\neq t_2$}
    \RightLabel{Eq-trans}
    \TrinaryInfC{$\Gamma$}
    \DisplayProof
    \\[2em]
    \AxiomC{$\Gamma,t_0=t_1$}
    \AxiomC{$\Gamma,s_0=s_1$}
    \AxiomC{$\Gamma,s_0\in t_0$}
    \AxiomC{$\Gamma,s_1\notin t_1$}
    \RightLabel{Eq-rel}
    \QuaternaryInfC{$\Gamma$}
    \DisplayProof
    &
    \AxiomC{$\Gamma,s=t$}
    \AxiomC{$\Gamma,\sfj(s)\neq \sfj(t)$}
    \RightLabel{Eq-ftn}
    \BinaryInfC{$\Gamma$}
    \DisplayProof
    \\[2em]
    \AxiomC{$\Gamma,\lnot A$}
    \RightLabel{$\mathrm{Ax}_\ZFC$ ($A$ is an axiom of $\ZFC$)}
    \UnaryInfC{$\Gamma$}
    \DisplayProof
    &
    \AxiomC{$\Gamma,A(t)$}
    \AxiomC{$\Gamma,\lnot A(\sfj(t))$}
    \RightLabel{Elem ($A\in \Sigma_\infty$)}
    \BinaryInfC{$\Gamma$}
    \DisplayProof
    \\[2em]
    \AxiomC{$\Gamma,\sfj(\dot{\kappa})=\dot{\kappa}$}
    \RightLabel{Nontriv}
    \UnaryInfC{$\Gamma$}
    \DisplayProof
    &
    \AxiomC{$\Gamma,\exists x\in t (A(x)\land x\notin y),y\nsubseteq t, \exists x\in y \lnot A(x)$}
    \RightLabel{$\Delta^j_0$-Sep}
    \UnaryInfC{$\Gamma$}
    \DisplayProof
\end{longtable}
\end{center}
In $b\forall$ and $\Delta^j_0$-Sep, $y$ does not occur free in the lower sequent, $s\subseteq t$ is an abbreviation of $\forall x\in s(x\in t)$, and in $\Delta^j_0$-Sep, $A(x)$ is a $\Delta^j_0(\Sigma_\infty)$-formula.
We call the free variable $y$ the \emph{eigenvariable} of the corresponding inference rule.
By \autoref{Theorem: Sect4-Expanded cutfree completeness}, every $\ZFC+\WA_0$-theorem is a theorem of $\bfG_{\WA_0}$.
Also, observe that for every $\bfG_{\WA_0}$-proof $\pi$ with the end sequent $\Gamma$, every formula occurring in $\pi$ is either $\Delta^j_0(\Sigma_\infty)$-formula or a subformula of a formula in $\Gamma$.

The presence of a partial truth predicate shows the soundness of $\ZFC+\WA_0$. More precisely,
\begin{theorem}[$\ZFC+\WA_0$ + $\Pi^j_1$-Induction] \label{Theorem: Sect4-GWA0 proof soundness}
    Suppose that $\bfG_{\WA_0}\vdash \Gamma(\vec{x})$, where every formula of $\Gamma(\vec{x})$ is $\Delta^j_0(\Sigma_\infty)$ with free variables $\vec{x}$. Then $\forall \vec{a}\in V^{<\omega }\left[\vDash_{\Delta^j_0(\Sigma_\infty)} \bigvee \Gamma(\vec{a})\right]$ if we interpret $\sfj$ by $j$, and $\dot{\kappa}$ by $\kappa=\operatorname{crit} j$.%
    \footnote{We defined $\vDash_{\Delta^j_0(\Sigma_\infty)}$ for formulas in the language without $\dot{\kappa}$. One way to get around this issue is to replace $\dot{\kappa}$ to a fixed free variable, then substitute it with $\kappa$.} 
\end{theorem}
\begin{proof}
    Let us fix a $\bfG_{\WA_0}$-proof $\pi$ with end sequent $\Gamma=\Gamma(\vec{x})$.
    $\pi$ is a finitely branching tree, so it must be finite by K\H onig's lemma.
    Hence we can find a finite enumeration $\langle \Gamma_i \mid i<m\rangle$ of every sequent occurring in $\pi$ such that if $\Gamma_i$ occurs as a hypothesis over $\Gamma_j$ in $\pi$, then $i<j$. (In particular, $\Gamma_0$ must follow from Ax.) The previous observation shows that every formula in $\Gamma_i$ is $\Delta^j_0(\Sigma_\infty)$.
    
    Now we can prove $\forall \vec{a}\in V^{<\omega }\left[\vDash_{\Delta^j_0(\Sigma_\infty)} \bigvee \Gamma_i(\vec{a})\right]$ by $\Pi^j_1$-Induction on $i$:
    The critical but tedious step is that for each inference rule of $\bfG_{\WA_0}$ in which every occurring formula is $\Delta^j_0(\Sigma_\infty)$, if every upper sequent is true under $\vDash_{\Delta^j_0(\Sigma_\infty)}$, then the lower sequent is also true under $\vDash_{\Delta^j_0(\Sigma_\infty)}$.
    The most non-trivial cases are when the rule is either $\exists$ or $\forall$. In this case, however, the principal formula of the inference rule is $\Sigma_\infty$ so that we can proceed with the argument.
\end{proof}

Hence we have the following:
\begin{corollary} \label{Corollary: ZFC WA0 consistency} \pushQED{\qed}
    $\ZFC+\WA_0$ + $\Pi^j_1$-Induction proves, hence $\ZFC+\WA_1$ proves the following:
    \begin{enumerate}
        \item $\ZFC+\WA_0$ is $\Delta^j_0(\Sigma_\infty)$-sound: That is, every $\Delta^j_0(\Sigma_\infty)$-theorem of $\ZFC+\WA_0$ is true.
        \item $\ZFC+\WA_0$ is consistent. \qedhere 
    \end{enumerate}
\end{corollary}

\subsection{Wholeness by numbers, one, two, three...}
In this subsection, we prove that $\ZFC+\WA_{n+1}$ proves the consistency of $\ZFC+\WA_n$. 
Let us consider the sequent calculus $\bfG_{\WA_n}$ obtained from $\bfG_{\WA_0}$ by replacing $\Delta^j_0$-Sep rule to $\Sigma^j_n$-Sep rule, which is obtained from $\Delta^j_0$-Sep rule by allowing $A(x)$ to be $\Sigma^j_n(\Sigma_\infty)$.
Observe that $\ZFC + \WA_0$ + $\Pi^j_1$-Induction proves $\vDash_{\Sigma^j_n(\Sigma_\infty)}$ is $\Sigma^j_n$-definable, so so we may replace the $\Sigma^j_n$-Sep rules to its single instance given by $\vDash_{\Sigma^j_n(\Sigma_\infty)}$.
Moreover, a proof of \autoref{Lemma: SimplifySeparation} tells us we may assume that $A(x)$ has a unique single variable $x$.
Again, it follows that if $\ZFC+\WA_n\vDash \bigvee\Gamma$, then $\bfG_{\WA_n}\vdash \Gamma$.

Now we claim that $\ZFC+\WA_n + \Pi^j_{n+1}$-Induction proves the consistency of $\ZFC+\WA_n$.
More precisely, we have the following:
\begin{theorem}[$\ZFC+\WA_n + \Pi^j_{n+1}$-Induction]
    Let $\Gamma$ be a sequent in which every formula is in $\Sigma^j_n(\Sigma_\infty)\cup \Pi^j_n(\Sigma_\infty)$.
    We define $\Gamma_{\Sigma^j_n(\Sigma_\infty)} \subseteq \Gamma$ is the subsequent comprising every $\Sigma^j_n(\Sigma_\infty)$-formula of $\Gamma$, and define $\Gamma_{\Pi^j_n(\Sigma_\infty)} \subseteq \Gamma$ similarly.
    Suppose that $\bfG_{\WA_n}\vdash\Gamma$. Then we have 
    \begin{equation*} \textstyle
        \forall \vec{a}\in V^{<\omega}\left[
        \vDash_{\underline{\Sigma}^j_n(\Sigma_\infty)} \left( \bigvee \Gamma_{\Sigma^j_n(\Sigma_\infty)} \right)(\vec{a}) \lor \vDash_{\underline{\Pi}^j_n(\Sigma_\infty)} \left( \bigvee \Gamma_{\Pi^j_n(\Sigma_\infty)} \right)(\vec{a}) \right]
    \end{equation*}
    if we interpret $\sfj$ by $j$ and $\dot{\kappa}$ by $\kappa$.
\end{theorem}
\begin{proof}
    The proof is more convoluted than that of \autoref{Theorem: Sect4-GWA0 proof soundness} due to $\Sigma^j_n$-Sep rule:
    The upper sequent of $\Sigma^j_n$-Sep rule has $\exists x\in y \lnot A(x)$, which is not in $\Sigma^j_n(\Sigma_\infty)\cup \Pi^j_n(\Sigma_\infty)$ when $A$ is $\Sigma^j_n(\Sigma_\infty)$.
    We handle this issue by manually substituting an appropriate instance.
    
    Let us fix a $\bfG_{\WA_n}$-proof $\pi$ of $\Gamma$.
    From the previous discussion, we may assume that $\pi$ uses a single instance of $\Sigma^j_n$-Sep for the formula $A(x)$ whose only free variable is $x$.
    Let $\{x_i\mid i<N\}$ be the enumeration of eigenvariables of $\forall$ or $b\forall$ occurring in $\pi$, and $\{y_i\mid i<m\}$ and $\{t_i\mid i<m\}$ be the enumeration of eigenvariables and terms of $\Sigma^j_n$-Sep rule
    \begin{center}
        \AxiomC{$\Gamma,\exists x\in t_i (A(x)\land x\notin y_i),y_i\nsubseteq t_i, \exists x\in y_i \lnot A(x)$}
        \RightLabel{$\Sigma^j_n$-Sep}
        \UnaryInfC{$\Gamma$}
        \DisplayProof
    \end{center}
    By indexing the eigenvariables from the bottommost sequent of $\pi$, we may assume that if $y_{i'}$ occurs in $\Gamma$ of the $\Sigma^j_n$-Sep rule, then $i'<i$.
    We can also find a finite sequence of natural number $\{k_i\mid i<m\}$ such that if $x_k$ occurs in $\Gamma$, then $k<k_i$. We may also assume $\langle k_i\mid i<m\rangle$ is increasing. 
    Then we have that every $t_i$ is of the form $\sfj^{p_i}(v)$ for some $p_i<\omega$ and $v\in \{\dot{\kappa}\}\cup \{x_k\mid k<k_i\}\cup \{y_{i'}\mid i'<i\}$.
    
    Suppose we are given a finite sequence of sets $\vec{a}$.
    Then let us define subsidiary parameters $c_i(\vec{a})$ as follows:
    \begin{itemize}
        \item $c_i(\vec{a}) := j^{p_i}(\kappa)$ when $t_i \equiv \sfj^{p_i}(\dot{\kappa})$.
        \item $c_i(\vec{a}) := j^{p_i}(\vec{a}_{x_k})$ when $t_i \equiv \sfj^{p_i}(x_k)$.
        \item $c_i(\vec{a}) := j^{p_i}(c_{i'}(\vec{a}))$ when $t_i \equiv \sfj^{p_i}(y_{i'})$.
    \end{itemize}
    We set $c(\vec{a}) = \bigcup_{i<m} c_i(\vec{a})$ and $b(\vec{a}) = \{x\in c(\vec{a}) \mid \vDash_{\Sigma^j_n(\Sigma_\infty)} A(x)\}$.
    Then define
    \begin{equation*}
        d_i(\vec{a}) = 
        \begin{cases}
            j^{p_i}(b_{i'}(\vec{a})), & t_i\equiv \sfj^{p_i}(y_{i'}), \\ 
            c_i(\vec{a}), & \text{otherwise,}
        \end{cases}
    \end{equation*}
    and $b_i(\vec{a}) = b(\vec{a}) \cap d_i(\vec{a})$.
    Then we have the following:
    \begin{enumerate}
        \item $\vDash_{\Delta^j_0(\Sigma_\infty)} d_i(\vec{a}) = t_i[\vec{a}]$, 
        \item $d_i(\vec{a})\cup b_i(\vec{a})\subseteq c_i(\vec{a})$, and
        \item $b_i(\vec{a}) = \{x\in d_i(\vec{a}) \mid \vDash_{\Sigma^j_n(\Sigma_\infty)} A(x)\}$.
    \end{enumerate}
    Finally, define $\vec{a}' = \vec{a}[b_0(\vec{a})/y_0,\cdots, b_{m-1}(\vec{a})/y_{m-1}]$.
    
    Before applying the induction, we further assume we fixed an enumeration of sequents occurring in $\pi$ in a way that upper sequents occur earlier than lower sequents in the enumeration. 
    There is no guarantee that every formula occurring in $\pi$ is in $\Sigma^j_n(\Sigma_\infty)\cup \Pi^j_n(\Sigma_\infty)$.
    However, we can still see that the only formula occurring in $\pi$ that is not in $\Sigma^j_n(\Sigma_\infty)\cup \Pi^j_n(\Sigma_\infty)$ is $\exists x\in y_i \lnot A(x)$ for some eigenvariable $y_i$.
    Then we claim the following by $\Pi^j_{n+1}$-Induction on a sequent $\Gamma$ of $\pi$:
    \begin{equation} \label{Equation: Sect4-Sequent Validity} \textstyle
        \vDash_{\underline{\Sigma}^j_n(\Sigma_\infty)} \left( \bigvee \Gamma_{\Sigma^j_n(\Sigma_\infty)} \right)(\vec{a}') \lor \vDash_{\underline{\Pi}^j_n(\Sigma_\infty)} \left( \bigvee \Gamma_{\Pi^j_n(\Sigma_\infty)} \right)(\vec{a}').
    \end{equation}
    ($\Gamma$ may contain formulas not in $\Sigma^j_n(\Sigma_\infty)\cup \Pi^j_n(\Sigma_\infty)$, but it will turn out that we can ignore such formulas.)
    Let us only consider the following non-trivial cases:
    \begin{enumerate}
        \item When $\Gamma$ is the lower sequent of $\Sigma^j_n$-Sep rule:
        By the induction hypothesis, we have either \eqref{Equation: Sect4-Sequent Validity}, or $\exists x\in d_i(\vec{a}) [\vDash_{\Sigma^j_n(\Sigma_\infty)} A(x) \land x\notin b_i(\vec{a})]$, or $b_i(\vec{a})\nsubseteq d_i(\vec{a})$. The latter two cases fail, so we get \eqref{Equation: Sect4-Sequent Validity}.

        \item When $\Gamma$ is the lower sequent of $b\exists$ rule introducing $\exists x\in y_i \lnot A(x)$: 
        \begin{center}
            \AxiomC{$\Gamma,\exists x\in y_i \lnot A(x),t\in y_i$}
            \AxiomC{$\Gamma,\exists x\in y_i \lnot A(x),\lnot A(t)$}
            \RightLabel{$b\exists$}
            \BinaryInfC{$\Gamma,\exists x\in y_i \lnot A(x)$}
            \DisplayProof
        \end{center}
        Let $e=t(\vec{a}')$.
        From the induction hypothesis, we have (\eqref{Equation: Sect4-Sequent Validity} or $e\in b_i(\vec{a})$) and (\eqref{Equation: Sect4-Sequent Validity} or $\vDash_{\Pi^j_n(\Sigma_\infty)}\lnot A(e)$).
        But we have $\lnot( e\in b_i(\vec{a})\land \vDash_{\Pi^j_n(\Sigma_\infty)} \lnot A(e))$, so \eqref{Equation: Sect4-Sequent Validity} holds.
        \qedhere 
    \end{enumerate}
\end{proof}

Hence we have
\begin{corollary} \label{Corollary: ZFC WAn consistency} \pushQED{\qed}
    $\ZFC + \WA_n + \Pi^j_{n+1}$-Induction proves the following:
    \begin{enumerate}
        \item $\ZFC+\WA_n$ is sound for Boolean combinations of $\Sigma^j_n(\Sigma_\infty)$-sentences.
        \item $\ZFC+\WA_n$ is consistent. \qedhere 
    \end{enumerate}
\end{corollary}

We also note that the consistency separation result also shows $\ZFC+\WA$ is not finitely axiomatizable:
\begin{corollary} \label{Corollary: Sect4-WAnotfinite}
    $\ZFC+\WA$ is \emph{not} finitely axiomatizable. 
\end{corollary}
\begin{proof}
    If it were, then $\ZFC+\WA$ reduces to $\ZFC+\WA_n$ for some $n$, so $\ZFC+\WA_n$ proves the consistency of itself, a contradiction.
\end{proof}

\subsection{Separating $\Pi^j_n$-Induction}
$\Pi^j_{n+1}$-Induction has a critical role in the previous subsections, and we may ask if we can separate $\Pi^j_n$-Induction by their strength.
To answer the question, we prove that for $m\ge n$, $\ZFC+\WA_n + \Pi^j_{m+1}$-Induction proves the consistency of  $\ZFC+\WA_n + \Pi^j_{m}$-Induction. We fix $m\ge n$ throughout this section, and we may assume $m>n$ since we have considered the case $m=n$.

We need an appropriate sequent calculus for $\ZFC+\WA_n + \Pi^j_{m}$-Induction to prove the theorem.
Let us recall that the statement $x\in\omega$ is $\Delta_0$-expressible in set theory and so we identify $x\in\omega$ with the corresponding $\Delta_0$-expression.
Then let us consider the following rule: 
\begin{center}
    \AxiomC{$\Gamma,A(0)$}
    \AxiomC{$\Gamma,y\notin \omega,\lnot A(y),A(y+1)$}
    \AxiomC{$\Gamma,t\notin \omega,A(t)$}
    \RightLabel{$\Pi^j_m$-Ind}
    \TrinaryInfC{$\Gamma$}
    \DisplayProof
\end{center}
Here $A(x)$ is $\Pi^j_m(\Sigma_\infty)$, and $y$ does not occur free in $\Gamma$. 
In principle, $0$ and $y+1$ are not part of the language of set theory we consider. However, both $0$ and the successor function are $\Delta_0$-definable, and for each $\Sigma^j_m(\Sigma_\infty)$-formula $A(x)$, we can uniformly find a $\Sigma^j_m(\Sigma_\infty)$-formula equivalent to $A(0)$ and $A(y+1)$ respectively. We fix such $\Sigma^j_m(\Sigma_\infty)$-formulas, and apply them to formulate $\Sigma^j_m$-Ind rule.

We can see that if $\ZFC+\WA_n$ + $\Pi^j_m$-Induction proves $\phi$, then $\bfG_{\WA_n}+\text{$\Pi^j_m$-Ind}\vdash\phi$.
Then let us prove the following:
\begin{proposition}[$\ZFC+\WA_n+\Pi^j_{m+1}$-Induction]
    Let $\Gamma$ be a sequent in which every formula is $\Sigma^j_m(\Sigma_\infty)\cup \Pi^j_m(\Sigma_\infty)$.
    If $\bfG_{\WA_n}+\Pi^j_m\mhyphen \mathrm{Ind}\vdash \Gamma$, then the following holds:
    \begin{equation*} \textstyle
        \forall \vec{a}\in V^{<\omega}\left[
        \vDash_{\underline{\Sigma}^j_m(\Sigma_\infty)} \left( \bigvee \Gamma_{\Sigma^j_n(\Sigma_\infty)} \right)(\vec{a}) \lor \vDash_{\underline{\Pi}^j_m(\Sigma_\infty)} \left( \bigvee \Gamma_{\Pi^j_n(\Sigma_\infty)} \right)(\vec{a}) \right].
    \end{equation*}
\end{proposition}
\begin{proof}
    Fix a $\bfG_{\WA_n}+\Pi^j_m\mhyphen \mathrm{Ind}$-proof $\pi$ of $\Gamma$.
    We can see that every formula occurring in $\pi$ is in $\Sigma^j_m(\Sigma_\infty)\cup \Pi^j_m(\Sigma_\infty)$. (We assume $m>n$, so $\Sigma^j_n$-Sep does not cause any complexity issues.)
    Then the desired result follows from the induction argument.
\end{proof}

\begin{corollary} \pushQED{\qed}
    Let $m\ge n$.
    $\ZFC + \WA_n + \Pi^j_{m+1}$-Induction proves the following:
    \begin{enumerate}
        \item $\ZFC+\WA_n$ is sound for Boolean combinations of $\Sigma^j_m(\Sigma_\infty)$-sentences.
        \item $\ZFC+\WA_n + \Pi^j_m$-Induction is consistent. \qedhere 
    \end{enumerate}
\end{corollary}

\printbibliography

\end{document}